\documentclass[11pt]{amsart}

\usepackage{epsf,amssymb,amsmath,overpic,amscd,psfrag,stmaryrd,times,color,mathrsfs,euscript}
\usepackage[all]{xy}
\usepackage{hyperref}
\usepackage{marginnote}
\usepackage{graphicx}
\usepackage{enumitem}
\usepackage{amsfonts}
\usepackage{amsmath}
\usepackage{soul}
\usepackage{graphicx}

\usepackage{tikz}
\usetikzlibrary{decorations.pathreplacing,angles,quotes}
\usetikzlibrary{arrows.meta}

\usepackage{accents}

\newtheorem{thm}{Theorem}[section]
\newtheorem{prop}[thm]{Proposition}
\newtheorem{lemma}[thm]{Lemma}
\newtheorem{cor}[thm]{Corollary}

\newtheorem{claim}[thm]{Claim}

\numberwithin{equation}{section}
\numberwithin{thm}{section}

\theoremstyle{definition}
\newtheorem{defn}[thm]{Definition}

\newtheorem{condition}[thm]{Condition}

\theoremstyle{remark}

\newtheorem{rmk}[thm]{Remark}

\newtheorem*{thm*}{Theorem}

\DeclareMathAlphabet\mathbfcal{OMS}{cmsy}{b}{n}
\DeclareMathAlphabet{\mathpzc}{OT1}{pzc}{m}{it}

\newcommand{\C}{\mathbb{C}}

\newcommand{\R}{\mathbb{R}}
\newcommand{\Z}{\mathbb{Z}}

\newcommand{\s}{\vskip.1in}
\newcommand{\n}{\noindent}

\newcommand{\be}{\begin{enumerate}}
\newcommand{\ee}{\end{enumerate}}
\newcommand{\op}{\operatorname}

\newcommand{\del}{\partial} 
\newcommand{\pd}[2]{\frac{\del #1}{\del #2}}

\newcommand{\cb}{\color{black}}
\newcommand{\cbu}{\color{black}}

\newcommand{\cpp}{\color{black}}
\newcommand{\ind}{\text{ind}}

\makeatletter
\newcommand*{\da@rightarrow}{\mathchar"0\hexnumber@\symAMSa 4B }
\newcommand*{\da@leftarrow}{\mathchar"0\hexnumber@\symAMSa 4C }
\newcommand*{\xdashrightarrow}[2][]{%
  \mathrel{%
    \mathpalette{\da@xarrow{#1}{#2}{}\da@rightarrow{\,}{}}{}%
  }%
}
\newcommand{\xdashleftarrow}[2][]{%
  \mathrel{%
    \mathpalette{\da@xarrow{#1}{#2}\da@leftarrow{}{}{\,}}{}%
  }%
}
\newcommand*{\da@xarrow}[7]{%
  \sbox0{$\ifx#7\scriptstyle\scriptscriptstyle\else\scriptstyle\fi#5#1#6\m@th$}%
  \sbox2{$\ifx#7\scriptstyle\scriptscriptstyle\else\scriptstyle\fi#5#2#6\m@th$}%
  \sbox4{$#7\dabar@\m@th$}%
  \dimen@=\wd0 %
  \ifdim\wd2 >\dimen@
    \dimen@=\wd2 %
  \fi
  \count@=2 %
  \def\da@bars{\dabar@\dabar@}%
  \@whiledim\count@\wd4<\dimen@\do{%
    \advance\count@\@ne
    \expandafter\def\expandafter\da@bars\expandafter{%
      \da@bars
      \dabar@
    }%
  }%
  \mathrel{#3}%
  \mathrel{%
    \mathop{\da@bars}\limits
    \ifx\\#1\\%
    \else
      _{\copy0}%
    \fi
    \ifx\\#2\\%
    \else
      ^{\copy2}%
    \fi
  }%
  \mathrel{#4}%
}

\newcommand{\sett}[1]{\{\,#1\,\}}

\begin{document}

\title{Immersed Lagrangian Floer cohomology via pearly trajectories}
\author{Garrett Alston}
\author{Erkao Bao}
\address{Erkao Bao, School of Mathematics, University of Minnesota, 127 Vincent Hall, 206 Church St. SE, Minneapolis, MN 55455}
\maketitle
\begin{abstract}
We define Lagrangian Floer cohomology \cbu over $\Z_2$-coefficients \cb by counting {\em pearly trajectories} for graded, exact Lagrangian immersions that satisfy a certain positivity condition on the index of the non-embedded points,
and show that it is an invariant of the Lagrangian immersion under Hamiltonian
deformations. We also show that it is naturally isomorphic to the Hamiltonian perturbed version
of Lagrangian Floer cohomology as defined in \cite{alston2018exact}. 
As an application, we prove that the number of non-embedded points of such a Lagrangian in $\C^n$ is no less than the sum of its Betti numbers.
\end{abstract}


\section{Introduction}
\label{sec:introduction}
Immersed Lagrangian Floer theory was first studied by Akaho in \cite{MR2155230} with a topological condition (second relative homotopy group vanishes) to preclude disc bubbling.
Later, Akaho and Joyce developed the theory in complete generality in \cite{MR2785840}, using the method of Kuranishi structures introduced in \cite{fukaya1999arnold, fooo} to deal with disc bubbling.
In  \cite{alston2018exact}, the authors developed a Floer theory for graded, exact, Lagrangian immersions that satisfy a certain positivity condition.
The theory is based on the Hamiltonian perturbation approach: the chain complex is generated by Hamiltonian chords, and the differential counts Hamiltonian perturbed holomorphic strips.
The positivity condition a priori eliminates the disc bubbles that cannot be easily handled by perturbing the almost complex structures.

The goal of this paper is to define a pearly version of Floer cohomology for a graded, exact, Lagrangian immersion $\iota:L\to M$ that satisfies a stronger positivity condition.
We'll explain in Section~\ref{subsection: disc bubble} (especially Claim~\ref{claim: constant strip index}) why the stronger positivity condition is needed.
The chain complex is generated by the critical points of a Morse function on $L$ and elements of $$R:=\left\{ (p,q)\in L\times L|\iota(p)=\iota(q),p\neq q\right\}.$$
The differential counts pearly trajectories, which are mixtures of Morse gradient trajectories and holomorphic strips. 
Pearly trajectories were first introduced by Oh \cite{oh1996relative} and were later studied intensively by Biran and Cornea \cite{MR2555932, biran4221quantum} in the realm of embedded Lagrangians. 
The pearly version of Lagrangian Floer cohomology is more computable compared to the Hamiltonian perturbed version or the Kuranishi perturbed version (See Section~\ref{example} and \cite{2013arXiv1311.2327A}).

\subsection*{Setup}
Let $(M^{2n},\omega,\sigma,\Omega,J_M)$ be an exact graded
compact symplectic manifold with boundary.
Here $\omega=d\sigma$ is a symplectic form, $J_M$ is an almost complex structure that is compatible
with $\omega$, and $\Omega$ is
a nowhere vanishing section of $\Lambda^n_{\C}(T^*M, J_M)$. 
In fact, we only require that $\Omega$ to be defined up to a $\pm$ sign.
Following Section (7a) of \cite{seidel-fcpclt} we require $J_M$ to be convex near $\partial M$,
which is the standard way to study $J_M$-holomorphic
curves contained in $M$. 

Let $L^n$ be a smooth closed manifold,
and $\iota:L\to M$ be an exact Lagrangian immersion with transverse double points as the only non-embedded points.
Here, ``exact Lagrangian immersion'' means that there exists a function $h:L\to\mathbb{R}$ such that $\iota^{*}\sigma=dh$,
and ``transverse'' means,
\begin{displaymath}
d\iota(T_{p}L)+ d\iota(T_{q}L)=T_{\iota(p)}M,\quad\text{for all }(p,q)\in R.
\end{displaymath}
We also assume that $\iota$ is graded.
Namely, there exists a function \cbu
$\theta:L\to \R$ \cb such that $e^{2\pi i\theta}=\mathrm{Det}^{2},$ where
$\mathrm{Det}^{2}:L\to S^{1}$ is the phase function defined by $\mathrm{Det}^{2}=\iota^*(\Omega^{\otimes 2}/|\Omega|^{2}).$

\begin{defn}
  \label{dfn:action}
  The \textit{action} $\mathcal A(p,q)$ of $(p,q)\in R$ is defined to be
  \begin{displaymath}
    \mathcal A(p,q) = h(q)-h(p).
  \end{displaymath}
\end{defn}

\begin{rmk}
  \label{rmk:disc with branch jump}
  The action has the following important application: Let $D^2$ be the unit disc and $u:D^2\to M$ be a map that is continuous on all of $D^2$ and smooth everywhere except at the point $1\in\partial D^2$.
  Assume that there exists a map $\ell:[0,2\pi]\to L$ that lifts $u|_{\partial D^2}$ in the sense that $\iota(\ell(t)) = u(e^{it})$.
  Assume furthermore that $(p,q):=(\ell(0),\ell(2\pi))\in R$.
  See Figure \ref{fig:disc with branch jump}(a).
  We say that $u$ has a \textit{branch jump of type $(p,q)$}.
  Then
  \begin{displaymath}
    \int_{D^2} u^*\omega = \mathcal A(p,q).
  \end{displaymath}
  In particular, if $u$ is a $J_M$-holomorphic disc, then the symplectic area of $u$ is equal to the action $\mathcal A(p,q)$.
  In fact, if $u$ is a tree of holomorphic discs such that the boundary around the entire tree has a lift $\ell$ with a single branch jump of type $(p,q)$, then the symplectic area of the entire tree is again $\mathcal A(p,q)$.
  See Figure \ref{fig:disc with branch jump}(b).
\end{rmk}
\begin{rmk}
  In Remark \ref{rmk:disc with branch jump}, to use the terminology of \cite{alston2018exact}, we are considering $1\in\partial D^2$ an incoming marked point.
  If we viewed $1$ as an outgoing marked point, we would say it has a branch jump of type $(q,p)$.
  Below, when we consider holomorphic strips with marked points on the top and bottom boundaries, we will view the marked points as outgoing. 
  \cbu
  The basic fact is that an outgoing point of type $(p,q)$ can be attached to an incoming point of type $(p,q)$.
  \cb
  See Figure \ref{fig:branch jumps at node}.
\end{rmk}

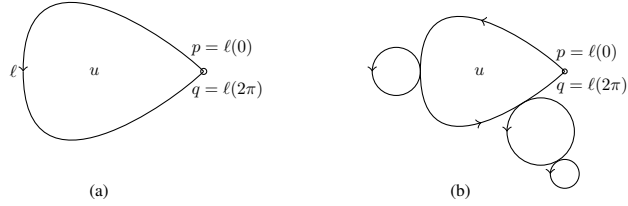
\begin{figure} 
\centering
\begin{tikzpicture}[scale = 0.8, every node/.style={scale=0.6}]
\draw[->] (2,0) .. controls (1.5, 0.5) and (-1, 2.3) .. (-1,0);
\draw (2,0) .. controls (1.5, -0.5) and (-1, -2.3) .. (-1,0);
\draw[] (2,0) circle [radius = 0.05];
\node[left] at (-1,0) {$\ell$};
\node[right] at (0, 0) {$u$};
\node[right] at (1.7, 0.4) {$p = \ell(0)$};
\node[right] at (1.7, -0.3) {$q = \ell(2\pi)$};

\usetikzlibrary{decorations.markings}
\begin{scope}[scale = 0.8, decoration={
    markings,
    mark=at position 0.5 with {\arrow{>}}}
    ] 
    \draw[] (10,0) circle [radius = 0.05];
    \node[right] at (8, 0) {$u$};
    \node[right] at (9.7, 0.4) {$p = \ell(0)$};
    \node[right] at (9.7, -0.3) {$q = \ell(2\pi)$};
    \draw[postaction={decorate}] (10,0) .. controls (9.5, 0.5) and (7, 2.3) .. (7,0);
    \draw[postaction={decorate}] (7,0) .. controls (7, -2.3) and (9.5, -0.5) .. (10,0);
    \draw[postaction={decorate}] (6.5,0) circle [radius = 0.5];
    \draw[postaction={decorate}] (9.5,-1.25) circle [radius = 0.7];    
    \draw[postaction={decorate}] (10,-2.13) circle [radius = 0.3];    
\end{scope}
\node[right] at (0,-2) {(a)};
\node[right] at (6,-2) {(b)};
\end{tikzpicture}
\caption[]{(a) Disc with a branch jump. (b) Tree of discs with a branch jump.}
\label{fig:disc with branch jump}
\end{figure}

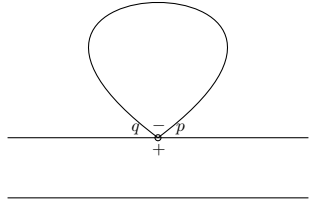
\begin{figure} 
\centering
\begin{tikzpicture}[scale = 0.8, every node/.style={scale=0.6}]
\draw (-2.5,0) -- (2.5,0);
\draw (-2.5, -1) -- (2.5, -1);
\draw (0,0) .. controls (-4, 3) and (4,3) .. (0,0);
\node[left] at (-0.2,0.15) {$q$};
\node[right] at (0.2,0.15) {$p$};
\node[right] at (-0.2,0.2) {$-$};
\node[right] at (-0.2,-0.2) {$+$};
\draw (0,0) circle [radius = 0.05];

\end{tikzpicture}
\caption[]{Nodal point of holomorphic curve indicating matching branch jumps. ``-'': incoming side; ``+'': outgoing side.}
\label{fig:branch jumps at node}
\end{figure}

\begin{defn}
  \label{defn:branch jump index}
  The \textit{index} $\ind(p,q)$ of $(p,q)\in R$ is defined to be
  \begin{equation}
    \label{index of p,q}
    \ind (p,q)=n+\theta(q)-\theta(p)-2\text{Angle}(d\iota(T_{p}L),d\iota(T_{q}L)).
  \end{equation}
Here $\text{Angle}(d\iota(T_{p}L),d\iota(T_{q}L))$ is the K\"ahler angle
between $d\iota(T_{p}L)$ and $d\iota(T_{q}L)$, which is defined as follows:
choose a unitary basis $(u_1,..., u_n)$ of $d\iota(T_{p}L)$ such that $d\iota(T_qL) = \text{span}_{\R} \{e^{2\pi i \alpha_1}u_1,..., e^{2\pi i \alpha_n}u_n\}$ with $\alpha_i\in (0,1/2)$.
Then
\begin{displaymath}
  \text{Angle}(d\iota(T_{p}L),d\iota(T_{q}L)) = \alpha_1 + \cdots + \alpha_n.
\end{displaymath}
It is easy to check that $\ind(p,q) = n - \ind(q,p)$.
\end{defn}

\subsection*{Main results}
Before stating the main results of this paper, we review one of the main results of  \cite{alston2018exact}.
There, the authors study $HF_H^*(\iota)$ the Hamiltonian perturbation version of Floer cohomology of a Lagrangian immersion $\iota$.
Its cochain complex $CF_H^*(\iota)$ is the free $\Z_2$-module generated by the Hamiltonian chords of a generic time-dependent Hamiltonian $H:[0,1]\times M\to\mathbb{R}$ with end points in $\iota(L)$, and the differential $d_H$ counts inhomogeneous holomorphic strips connecting Hamiltonian chords.
The strips $u$ satisfy the equation
\begin{equation}
  \label{eq:1}
  \frac{\partial u}{\partial s} + J(u)\left(\frac{\partial u}{\partial t} - X_{H}(u)\right) = 0.
\end{equation}
Here $s+it$ are the natural complex coordinates on the strip $\mathbb{R}\times[0,1]$ and $X_{H}$ is the Hamiltonian vector of $H$,
and $J$ is a generic $t$-dependent compatible almost complex structure on $M$ that equals $J_M$ near the end of $M$.
In order to rule out problematic disc bubbles, a positivity condition is assumed.
We will refer to this condition as weak positivity:
\begin{condition}[weak positivity]
If $(p,q)\in R$ and $\mathcal A(p,q)>0$, then $\ind(p,q)\geq3.$
\end{condition}

This condition implies that if a holomorphic disc bubbles off from a sequence of strips with a fixed Maslov index, then the strip component in the limit must have Maslov index at least 3 less.
See Figure \ref{fig:bubble under weak positivity}.
Along with a transversality result asserting that strips with Maslov index less than 1 or 0 (depending on the context) do not exist, this implies that a limit of the sequence of strips cannot have disc bubbles.
This is the main ingredient in proof of the following theorem:
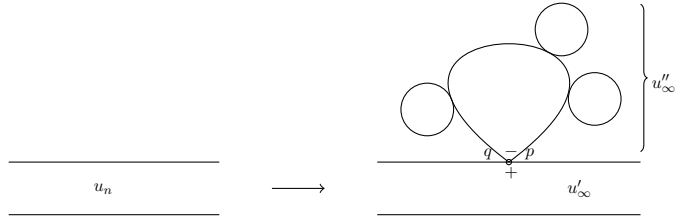
\begin{figure} 
\centering
\begin{tikzpicture}[scale = 0.7, every node/.style={scale=0.6}]
\draw (-9.5,0) -- (-5.5,0);
\node[right] at (-8, -0.5) {$u_n$};
\node[right] at (1, -0.5) {$u_{\infty}'$};
\draw (-9.5, -1) -- (-5.5, -1);
\draw (-2.5,0) -- (2.5,0);
\draw[->]  (-4.5, -0.5) -- (-3.5, -0.5);
\draw (-2.5, -1) -- (2.5, -1);
\draw (0,0) .. controls (-4, 3) and (4,3) .. (0,0);
\node[left] at (-0.2,0.15) {$q$};
\node[right] at (0.2,0.15) {$p$};
\node[right] at (-0.2,0.2) {$-$};
\node[right] at (-0.2,-0.2) {$+$};
\node[right] at (2.6, 1.5) {$u_\infty ''$};
\draw (0,0) circle [radius = 0.05];
\draw (-1.55,1) circle [radius = 0.5];
\draw (1.63,1.2) circle [radius = 0.5];
\draw (1, 2.52) circle [radius = 0.5];
\draw[decoration={brace,mirror},decorate]
  (2.5,0.2) -- (2.5,3) ;
\end{tikzpicture}
\caption[]{Nodal point of holomorphic curve indicating matching branch jumps. $\mu(u_n) = K$ for all $n$, $\mu(u_\infty') \leq K-3$ and $\mu(u_\infty'') \geq 3$.}
\label{fig:bubble under weak positivity}
\end{figure}

\begin{thm}[Theorem 1.4 in \cite{alston2018exact}]
  \label{thm:hamiltonian hf}
Under the weak positivity condition,
the differential $d_H$ is well-defined, and satisfies $d_H^2 = 0$. 
Moreover, the immersed Lagrangian Floer cohomology defined by $HF_H^*(\iota) := \ker d_H / \op{im}{d_H}$ is an invariant of $\iota$ under Hamiltonian isotopy.
\end{thm}

We now turn to the main results in this paper.
First, we define a pearly version of the immersed Lagrangian Floer cohomology $HF^*(\iota)$. 
Roughly speaking, the cochain complex $CF^*(\iota)$ is the free $\Z_2$-module generated by $R$ and the critical points of a Morse function and the differential $d$ counts trajectories of index $1$ that are made of combinations of holomorphic strips and Morse gradient trajectories connecting the generators of $CF^*(\iota)$.
As before, we need to impose a positivity condition to ensure that $d^2=0$ and various ingredients are well-defined.
In order to avoid complicated perturbation techniques such as Kuranishi structures and to keep things as simple as possible, we impose a stronger condition than before:
\begin{condition}[strong positivity]
\label{strong positivity}
If $(p,q)\in R$ and $\mathcal A(p,q)>0$ then $\ind (p,q)\geq\max\{\frac{n+2}{2},3\}$.
\end{condition}
The strong positivity condition is used to rule out certain bad disc bubbles in degenerations of pearly trajectories.
These bad degenerations are unique to pearly trajectories and do not occur for the strips used to define the Hamiltonian version of Lagrangian Floer cohomology in Theorem \ref{thm:hamiltonian hf}.
See Section~\ref{sec:moduli-spaces} for more details.

\begin{thm}[Lemmas \ref{prop:d-is-defined.} and \ref{d square equals 0}  and Theorem \ref{thm: iso}]
Under the strong positivity condition, the pearly chain complex $(CF^*(\iota),d)$ is a chain complex ($d$ is well-defined and $d^2=0$).
Moreover, its cohomology $HF^*(\iota)$ is naturally isomorphic to the Hamiltonian version of Lagrangian Floer cohomology $HF_H^*(\iota)$ from Theorem \ref{thm:hamiltonian hf}. In particular, $HF^*(\iota)$ is an invariant of $\iota$ under Hamiltonian isotopy.
\end{thm}


The following corollaries should be viewed within the context of a result in \cite{MR2155230}: if $\pi_2(M, \iota(L)) = 0$, then $HF^*(\iota) \cong H^*(L; \Z_2) \oplus \Z_2R$.

\begin{cor} \label{rank inequality}
Suppose that $\iota$ is an exact, graded Lagrangian immersion that satisfies the strong positivity condition and whose non-embedded points are transverse and double.
The number of non-embedded points satisfies 
\begin{displaymath}
  |R|  \geq \sum_i \op{rank} H^i(L; \Z_2) - \sum_i \op{rank} HF^i(\iota).  
\end{displaymath}
\end{cor}

\begin{cor}
Suppose that $\iota$ is an exact, graded Lagrangian immersion in $\C^n$ that satisfies the strong positivity condition and whose immersed points are transverse and double. 
Then the number of immersed points satisfies 
\begin{displaymath}
  |R| \geq \sum_i \op{rank} H^i(L; \Z_2).  
\end{displaymath}
\end{cor}

\subsection*{Guide to the rest of the paper}
In Section \ref{sec:moduli-spaces} we define the moduli spaces needed to define the differential $d$ in the pearly chain complex $(CF^*(\iota),d)$.
We use the strong positivity Condition \ref{strong positivity} to prove some index inequalities that is used in Section \ref{lagrangian floer via pearly trajectories} to prove that $d$ is well-defined and $d^2=0$.
As mentioned above, $d$ counts index $1$ trajectories made of holomorphic strips and Morse trajectories that connect generators of $CF^*(\iota)$.

To begin with, we consider the five types of trajectories pictured in Figure \ref{fig:four types of trajectories}.
\begin{figure} 
\centering
\begin{tikzpicture}[scale = 0.7, every node/.style={scale=0.6}]

\draw[->] (0,9) -- (2.5,9);
\draw (2.5,9) -- (5,9);
\filldraw (0,9) circle [radius = 0.05];
\filldraw (5,9) circle [radius = 0.05];
\node[left] at (-0.2,9) {$x$};
\node[right] at (5.2, 9) {$y$};
\node[] at (7,9) {(a)};

\draw[->] (0,1+6) -- (2,1+6);
\draw (2,1+6) -- (3,1+6);
\draw (3, 1+6) .. controls (3, 2+6) and (4.5,1.5+6) .. (5, 1+6)
(3,1+6) .. controls (3,0+6) and (4.5, 0.5+6) .. (5,1+6);
\node[right] at (5,1+6+0.3) {$q$};
\node[right] at (5,1+6-0.3) {$p$};
\node[left] at (5,1+6) {$+$};
\node[left] at (-0.2,1+6) {$x$};
\draw (5, 1+6) circle [radius = 0.05];
\filldraw (0,1+6) circle [radius = 0.05];
\node[] at (7,7) {(b)};

\draw[->] (2,1+4) -- (4,1+4);
\draw (4,1+4) -- (5,1+4);
\draw (0, 1+4) .. controls (0.5, 1.5+4) and (2,2+4) .. (2, 1+4)
(0,1+4) .. controls (0.5,0.5+4) and (2, 0+4) .. (2,1+4);
\node[right] at (5.2,1+4) {$x$};
\node[left] at (0,1+4+0.3) {$q$};
\node[left] at (0,1+4-0.3) {$p$};
\node[right] at (0,1+4) {$-$};
\draw (0, 1+4) circle [radius = 0.05];
\filldraw[black] (5,1+4) circle [radius = 0.05];
\node[] at (7,5) {(c)};

\draw (0, 1+2) .. controls (0.5, 2+2) and (3.5+1,2+2) .. (4+1, 1+2)
(0,1+2) .. controls (0.5,0+2) and (3.5+1, 0+2) .. (4+1,1+2);
\node[left] at (0,1+2+0.3) {$s$};
\node[left] at (0,1+2-0.3) {$r$};
\node[right] at (0,1+2) {$-$};
\node[right] at (4+1,1+2+0.3) {$q$};
\node[right] at (4+1,1+2-0.3) {$p$};
\node[left] at (4+1,1+2) {$+$};
\draw (0, 1+2) circle [radius = 0.05];
\draw (4+1, 1+2) circle [radius = 0.05];
\node[] at (7,3) {(d)};

\draw[->] (2,1) -- (2.5,1);
\draw (2.5,1) -- (3,1);
\draw 
(0, 1) .. controls (0.5, 1.5) and (2,2) .. (2, 1)
(0, 1) .. controls (0.5,0.5) and (2, 0) .. (2,1)
(4-1, 1) .. controls (4-1, 2) and (5.5-1,1.5) .. (6-1, 1)
(4-1, 1) .. controls (4-1,0) and (5.5-1, 0.5) .. (6-1,1);
\node[right] at (6-1,1+0.3) {$q$};
\node[right] at (6-1,1-0.3) {$p$};
\node[left] at (6-1,1) {$+$};
\node[left] at (0,1+0.3) {$s$};
\node[left] at (0,1-0.3) {$r$};
\node[right] at (0,1) {$-$};
\draw (0, 1) circle [radius = 0.05];
\draw (6-1,1) circle [radius = 0.05];
\node[] at (7,1) {(e)};
\end{tikzpicture}

\caption[]{Five types of trajectories connecting generators of $CF^*(\iota)$. The strong positivity condition actually implies that the last type does not exist (with index 1).}
\label{fig:four types of trajectories}
\end{figure}
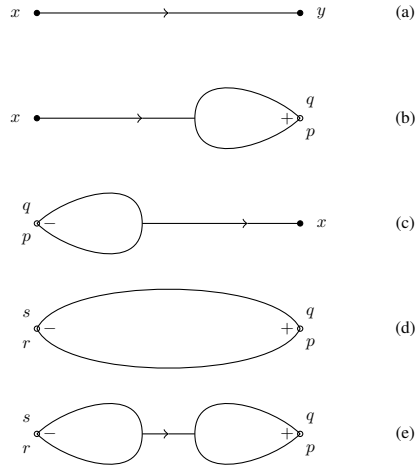
Type (a) corresponds to $\mathcal M(x,y)$ the moduli space  of Morse trajectories in $L$ from the critical point $x$ to the critical point $y$.
Type (d) corresponds to $\mathcal M_J((r,s), (p,q))$ the moduli space of $J$-holomorphic strips  from $(r,s)\in R$ to $(p,q) \in R$.
Type (b) corresponds to $ {\mathcal M}_J(x, (p,q))$ the space of pairs $(u,\ell)$ such that 
\begin{itemize}
\item $u:Z\to M$ is a $J$-holomorphic map such that $\lim _{s\to \pm \infty}u(s,t) = m_\pm$ for some $m_\pm \in M$ independent of $t$;
\item $\ell: \partial Z \to L$ is a continuous boundary lift of $u,$ i.e., $u|_{\partial Z}=\iota\circ\ell $ such that 
	\begin{itemize}
	\item  $\lim_{s\to -\infty}\ell(s,1) = \lim_{s\to -\infty}\ell(s,0) \in W^u(x)$, the unstable manifold of $x$ in $L$,
	\item $\lim_{s\to +\infty}\ell(s,1) = q$ and $\lim_{s\to +\infty}\ell(s,0) = p$.
	\end{itemize}
\item $(u,\ell)$ and $(u', \ell')$ are identified if they differ by an $\R$-translation in $Z$.
\end{itemize}
Type (c) moduli space denoted by $\mathcal M_J((p,q),x)$ can be defined similarly as the second one.
Type (e) actually does not show up in the definition of $d$ since if such a trajectory exists, the index of the trajectory equals $\ind(r,s) - \ind(p,q) \geq 2$ by strong positivity (see Remark \ref{lem:Under-strong-positivity} for details).

To show $d$ is well-defined, we prove the moduli space of $\text{index} = 1$ trajectories of the first four types in Figure \ref{fig:four types of trajectories} is a compact zero dimensional manifold. 
For compactness, one can use the Gromov's compactification to add all the possible degenerations to compactify the moduli space, and then ideally uses the transversality to show these degenerations do not exist. 
However, as usual, transversality in general is hard to achieve by simply perturbing the almost complex structure even in the exact Lagrangian immersion case. For this reason, we additionally assume the positivity conditions.
One particular type of degenerations that we use {\em strong} positivity condition to preclude are the ones that contain ``ghost'' (constant) components. 
An example of this is a constant strip from $(p,q)\in R$  to $(q,p)\in R$ with one holomorphic disc attached at each side of the strip pictured in Figure \ref{fig: ghost}.
 Using strong positivity, we have $\ind(p,q) - \ind(q,p) = 2 \ind(p,q) - n \geq 2 \cdot \frac{n+2}{2} -n = 2$.
 To show $d^2 = 0$ the additional ingredients we need are the standard gluing results.
 
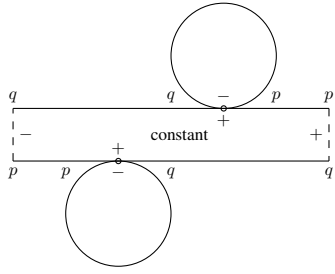
\begin{figure} 
\centering
\begin{tikzpicture}[scale = 0.7, every node/.style={scale=0.6}]
\draw (0,0.5) -- (6,0.5);
\draw[dashed] (6,0.5) -- (6, -0.5);
\draw (0,-0.5) -- (6,-0.5);
\draw[dashed](0,-0.5) -- (0,0.5);
\draw (4,1.5) circle [radius = 1];
\draw(4, 0.5) circle [radius = 0.05];
\draw(2, -1.5) circle [radius = 1];
\draw(2, -0.5) circle [radius = 0.05];
\node[right] at (2.5,0) {constant};
\node[above] at (0,0.5){$q$};
\node[above] at (3,0.5){$q$};
\node[above] at (5,0.5){$p$};
\node[above] at (6,0.5){$p$};
\node[below] at (0,-0.5){$p$};
\node[below] at (1,-0.5){$p$};
\node[below] at (3,-0.5){$q$};
\node[below] at (6,-0.5){$q$};
\node[right] at (0,0){$-$};
\node[left] at (6,0){$+$};
\node[above] at (4,0.5){$-$};
\node[below] at (4,0.5){$+$};
\node[below] at (2, -0.5){$-$};
\node[above] at (2, -0.5){$+$};
\end{tikzpicture}
\caption[]{domain of a ghost component with 2 discs attached}
\label{fig: ghost}
\end{figure}

In Section~\ref{lagrangian floer via hamiltonian perturbation} we show that the cohomology $HF^*(\iota)$ as above is isomorphic to the Hamiltonian version of Lagrangian Floer cohomology $HF^*_H(\iota)$. The isomorphism is a PSS (\cite{piunikhin1996symplectic}) type of isomorphism and is constructed via a chain map from $CF^*_H(\iota)$ to $CF^*(\iota)$ defined by counting three types of index $=0$ trajectories shown in Figure~\ref{fig: three trajectories} from the generators of $CF^*(\iota)$, which are self-intersection points $R$ and Morse critical points, to the generators of $CF^*_H(\iota)$, which are Hamiltonian chords. 
Let $\bold H := \{H_s\}_s$ be a $1$-parameter family of Hamiltonians with $\lim_{s\to -\infty}H_s = 0$ and $\lim_{s\to +\infty} H_s = H$. 
The three types of trajectories in Figure~\ref{fig: three trajectories} are the $\bold H$-perturbed version of Type (b), Type (d) and Type (e) in Figure~\ref{fig:four types of trajectories}, respectively.
In case (III) the perturbation applies only to the curve on the right-hand side.
Note that unlike the Type (e) in Figure~\ref{fig:four types of trajectories} the existence of Type (III) in Figure~\ref{fig: three trajectories} does not contradict to the strong positivity condition.


To show the chain map induces an isomorphism on homology, we construct a backward map similarly, and prove that their compositions are homotopic to the identity maps. 

In Section~\ref{example} we explicitly calculate $HF^*(\iota)$ for some exact immersed Lagrangian spheres inside the smoothing of an $A_N$ surface.

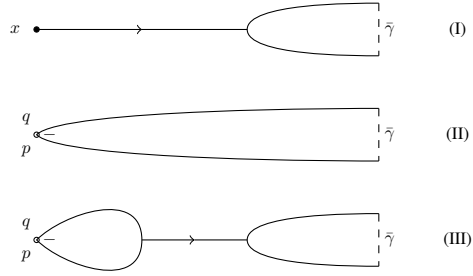
\begin{figure} 
\centering
\begin{tikzpicture}[scale = 0.7, every node/.style={scale=0.6}]
\draw[->] (0,1+4) -- (2,1+4);
\draw (2,1+4) -- (3+1,1+4);
\draw (3+1, 1+4) .. controls (3+1, 1.5+4) and (5+1,1.5+4) .. (5.5+1, 1.5+4)
(3+1,1+4) .. controls (3+1,0.5+4) and (5+1, 0.5+4) .. (5.5+1,0.5+4);
\draw[dashed] (5.5+1, 1.5+4) -- (5.5+1, 0.5+4);
\node[right] at (5.5+1,1+4) {$\bar\gamma$};
\node[left] at (-0.2,1+4) {$x$};
\node[] at (8, 5) {(I)};
\filldraw (0,1+4) circle [radius = 0.05];

\draw 
(0, 1+2) .. controls (0.5, 1.5+2) and (6,1.5+2) .. (6.5, 1.5+2)
(0, 1+2) .. controls (0.5,0.5+2) and (6, 0.5+2) .. (6.5,0.5+2);
\draw[dashed] (6.5, 1.5+2) -- (6.5, 0.5+2);
\node[right] at (6.5,1+2) {$\bar\gamma$};
\node[left] at (0,1+2.3) {$q$};
\node[left] at (0,1+1.7) {$p$};
\node[right] at (0,1+2) {$-$};
\node[] at (8,3) {(II)};
\draw (0, 1+2) circle [radius = 0.05];

\draw[->] (2,1) -- (3,1);
\draw (3,1) -- (4,1);
\draw 
(0, 1) .. controls (0.5, 1.5) and (2,2) .. (2, 1)
(0, 1) .. controls (0.5,0.5) and (2, 0) .. (2,1)
(4, 1) .. controls (4, 1.5) and (6,1.5) .. (6.5, 1.5)
(4,1) .. controls (4,0.5) and (6, 0.5) .. (6.5,0.5);
\draw[dashed] (6.5, 1.5) -- (6.5, 0.5);
\node[right] at (6.5,1) {$\bar\gamma$};
\node[left] at (0,1+0.3) {$q$};
\node[left] at (0,1-0.3) {$p$};
\node[right] at (0,1) {$-$};
\node[] at (8,1) {(III)};
\draw (0, 1) circle [radius = 0.05];
\end{tikzpicture}
\caption[]{Three trajectories that connect generators of $CF^*(\iota)$ to generators of $CF^*_H(\iota)$.}
\label{fig: three trajectories}
\end{figure}


%

\section{Moduli spaces} \label{sec:moduli-spaces}
To define $HF^{*}(\iota)$ we count  pearly trajectories that consist of a combination of Morse trajectories and homogeneous  (i.e., with no Hamiltonian perturbation) $J$-holomorphic strips.
To show that $HF^{*}(\iota)$ is isomorphic to $HF_H^{*}(\iota)$
we need the Hamiltonian perturbed $J$-holomorphic strips.
We handle these two cases together by allowing the Hamiltonian to be constantly zero.

Let $H$ be a $t$-dependent Hamiltonian functions on $M$.
Let $X_{H}$ be the Hamiltonian vector field of $H$ defined by $X_{H} \lrcorner  \omega = d H$.
Let $P_{H}$  be the set of pairs $\bar{\gamma} = (\gamma, \delta)$ where
\begin{itemize}
\item $\gamma: [0,1] \to M$ satisfies $\frac{d}{dt}\gamma (t) = X_{H}(\gamma(t))$ and $\gamma(0), \gamma(1) \in \iota(L)$,
\item $\delta: \{0,1\} \to L$ satisfies $\iota \circ \delta (i) = \gamma (i)$ for $i \in \{0,1\}$. 
\end{itemize}
Let $\phi_t^H$ be the flow of the Hamiltonian vector field $X_H$.
\begin{defn}
We say a Hamiltonian $H$ is {\em admissible} if
\begin{itemize}
\item $\phi_1^H(\iota(L))$ is transverse to $\iota(L)$, and
\item $(\phi_1^H)^{-1}(\iota(L)) \cup \phi_1^H(\iota(L)) \cap \iota(R) = \emptyset$.
\end{itemize}
\end{defn}
If $H$ is admissible, then given $\gamma$ a Hamiltonian chord of $H$, there is a unique lift $\bar\gamma \in P_H$.

{\em To combine notations, we also allow $H\equiv 0$, and in this case we require $\bar\gamma=(\gamma,\delta)$ to satisfy
$(\delta(0), \delta(1))\in R$, so $P_H$ is in one-to-one correspondence with $R$.
We remark that we $HF_{H}^{*}(\iota)$ is only defined for admissible Hamiltonians and $H\equiv 0$ is not admissible.
}

Let $J$ be a $t$-dependent compatible almost complex structure of $M,$ such that $J=J_{M}$ near the
end of $M$.
For $\bar \gamma_\pm \in P_{H}$,
we set $\widetilde{\mathcal M}_{{J}, {H}}(\bar\gamma_-, \bar\gamma_+; {\alpha})$ to be the space of tuples $(\Delta, u, \ell)$ such that 
\begin{enumerate}
\item $Z$ is the strip $\text{\ensuremath{\mathbb{R}}}\times[0,1]$
with coordinate $(s,t)$ and standard complex structure,
\item $\Delta=\{z_{1},z_{2},...,z_{k}\}\subset\mathbb{R}\times\{1\}\cup\mathbb{R}\times\{0\}$
is a set of distinct boundary marked points which we consider to be ordered in the natural way (counterclockwise around the boundary of the strip starting at $s=-\infty,t=0$),
\item $\bold{\alpha}:\{1,2,...,|\Delta|\}\to R$ is the map that specifies the type of branch jump. (We only consider outgoing marked points.)
\item map $u:Z\to M$ is continuous on $Z$ and differentiable in the interior $\mathring Z$,
\item map $\ell: \partial Z\backslash\Delta\to L$ is a continuous boundary lift of $u,$ i.e., $u|_{\partial Z\backslash\Delta}=\iota\circ\ell,$
\item for each $z_{i}\in\mathbb{R}\times\{1\},$ $\ell$ has the
branch jump of type $\bold{\alpha}(i)=(p,q)$, i.e., 
$$p=\lim_{s\to z_{i}+}\ell(s,1) \hspace{0.5cm} \text{ and } \hspace{0.5cm} q=\lim_{s\to z_{i}-}\ell(s,1);$$
 for each $z_{i}\in\mathbb{R}\times\{0\},$ $\ell$ has the
branch jump of type $\bold{\alpha}(i)=(p,q),$ i.e., 
$$p=\lim_{s\to z_{i}-}\ell(s,0) \hspace{0.5cm} \text{ and } \hspace{0.5cm} q=\lim_{s\to z_{i}+}\ell(s,0),$$
\item $u_{s}+J(u)(u_{t}-X_{H}(u))=0,$  
\item $\lim_{s\to\pm\infty}u(s,t)=\gamma_\pm(t)$ uniformly in $t$,
\item $\lim_{s\to \pm \infty} \ell(s,i) = \delta(i)$ for $i\in \{0,1\}$.
\end{enumerate}
\begin{rmk}
We say $u$ is a {\em trivial}, if $u(s,\cdot) = \gamma(\cdot)$ for all $s$.
In the case when $u$ is trivial  we require that $\Delta \neq \emptyset$.
When $\Delta = \emptyset$, we sometimes omit ${\alpha}$, and write $\widetilde{\mathcal M}_{{J}, {H}}(\bar \gamma_-, \bar\gamma_+)$. When $H \equiv 0$ we sometimes omit $H$ and write $\widetilde{\mathcal M}_{J}(\bar \gamma_-, \bar\gamma_+; \alpha)$. 
There is an $\R$-action on $\widetilde{\mathcal M}_{{J}, {H}}(\bar \gamma_-, \bar\gamma_+; \alpha)$, and we denote the quotient by ${\mathcal M}_{{J},{H}}(\bar \gamma_-, \bar \gamma_+;\alpha)$. 
\end{rmk}

When $H\equiv 0$, we also define two additional moduli spaces $\widetilde{\mathcal M}_J(\bar\gamma_-, \emptyset;\alpha)$ and $\widetilde{\mathcal M}_J(\emptyset, \bar\gamma_+;\alpha)$ by replacing the corresponding requirements in (8) and (9) with that 
$(u, \ell)$ has a removable singularity at $\{\pm \infty \} \times [0,1]$, respectively.

Let $f:L\to\mathbb{R}$ be a Morse function and $g$ a Riemannian metric on $L$. Denote by $\op{Crit}f$ be the set of critical points of 
$f$. For any $x \in \op{Crit}f$ and $\bar\gamma \in R$ we define the moduli space 
$${\mathcal M}_J(\bar\gamma, x; \alpha):= {\mathcal M}_J(\bar\gamma, \emptyset;\alpha) ~_{\op{ev}_{+\infty}} \!\times W^s(x),$$
where $W^s(x)$ is the stable manifold of $x$ (with respect to $-\nabla f$), and $\op{ev}_{+\infty}: {\mathcal M}_J(\bar\gamma, \emptyset;\alpha) \to L$ is the evaluation map at $+\infty$. See Figure~\ref{fig:four types of trajectories}(c) for the case when $\bar\gamma = (p,q)$.
Similarly, we also define the moduli space
$${\mathcal M}_J(x, \bar\gamma; \alpha):= W^u(x) \times_{\op{ev}_{-\infty}} {\mathcal M}_J(\emptyset, \bar\gamma;\alpha).$$
See Figure~\ref{fig:four types of trajectories}(b) for the case when $\bar\gamma = (p,q)$, and Figure~\ref{fig: three trajectories}(I) for the case when $\bar\gamma$ is a Hamiltonian chord of a non-zero $H$.

For the purpose of defining chain maps we also need to consider a smooth family of compatible almost complex structures $\bold J = \{J_s\}_s$ that agree with $J_M$ near the end, and a smooth family of Hamiltonian functions $\bold H = \{H_s\}_s$. 
We require that $\frac{\partial \bold J}{\partial s}$ and $\frac{\partial \bold H}{\partial s}$ vanishes when $|s|$ is sufficiently large.
We denote the corresponding moduli spaces by $\widetilde{\mathcal M}_{\bold{J}, \bold{H}}(\bar\gamma_-, \bar\gamma_+; {\alpha})$.
Note that here we allow $u$ to be a trivial map, and $\Delta = \emptyset$.
When $H_{-\infty} \equiv 0$ we can define $\widetilde{\mathcal M}_{\bold J, \bold H}(\emptyset, \bar \gamma; \alpha)$ and
\begin{equation}\label{eq:moduli space for chain map case I}
  \widetilde{{\mathcal M}}_{\bold J, \bold H}(x, \bar\gamma; \alpha):= W^u(x) \times_{\op{ev}_{-\infty}} \widetilde{\mathcal M}_{\bold J, \bold H}(\emptyset, \bar\gamma;\alpha).
\end{equation}
This is the modulie space of trajectories in case (I) in Figure \ref{fig: three trajectories}.
Similarly when $H_{+\infty} \equiv 0$, it makes sense to define
\begin{equation}\label{eq:moduli space for chain map case I opposite}
  \widetilde{{\mathcal M}}_{\bold J, \bold H}(\bar\gamma, x; \alpha):= \widetilde{\mathcal M}_{\bold J, \bold H}(\bar\gamma, \emptyset;\alpha) ~_{\op{ev}_{+\infty}} \!\times W^s(x).
\end{equation}
This is the moduli space of trajectories analogous to case (I) that will be used for the chain map in the opposite direction.
For case (II), the moduli spaces are
\begin{equation}\label{eq:moduli space for chain map case II}
  \widetilde{\mathcal M}_{\bold{J}, \bold{H}}(\bar\gamma', \bar\gamma; {\alpha}),\quad H_{-\infty}\equiv0,\quad\bar\gamma'\in R,\quad\bar\gamma\in P_{H_{+\infty}},
\end{equation}
and for the chain map in the opposite direction they are
\begin{equation}\label{eq:moduli space for chain map case II opposite}
  \widetilde{\mathcal M}_{\bold{J}, \bold{H}}(\bar\gamma, \bar\gamma'; {\alpha}),\quad H_{+\infty}\equiv0,\quad\bar\gamma'\in R,\quad\bar\gamma\in P_{H_{-\infty}}.
\end{equation}
For case (III), the moduli spaces are
\begin{multline}\label{eq:moduli space for chain map case III}
  \widetilde{\mathcal M}^{\text{pearl}}_{\bold J, \bold H}(\bar\gamma', \bar\gamma)=\{([\bar u],v)\in\mathcal M_J(\bar\gamma', \emptyset)\times\widetilde{\mathcal M}_{\bold J, \bold H}(\emptyset, \bar\gamma)\,|\\\,\varphi_L^\tau (\op{ev}_{+\infty}([\bar u])) = \op{ev}_{-\infty}(v),\text{ for some }s>0\}.
\end{multline}
Here $\varphi_L^\tau$ is the negative gradient flow of $f$. 
Similarly for the opposite chain map we have the moduli space
\begin{multline}\label{eq:moduli space for chain map case III opposite}
  \widetilde{\mathcal M}^{\text{pearl}}_{\bold J, \bold H}(\bar\gamma, \bar\gamma')=\{(v,[\bar u])\in\widetilde{\mathcal M}_{\bold J, \bold H}(\bar\gamma,\emptyset)\times\mathcal M_J(\emptyset,\bar\gamma')\,|\\\,\varphi_L^\tau (\op{ev}_{+\infty}(v)) = \op{ev}_{-\infty}(\bar u),\text{ for some }s>0\}.
\end{multline}

\subsection*{Energy}
\begin{defn}
For $\bar \gamma = (\gamma, \delta) \in P_H$, we define its action by  
$$\mathcal A_H(\bar \gamma)=-\int_{[0,1]} \left(\gamma^*\sigma+ H(\gamma(t))\right)dt-h(\delta(0))+h(\delta(1)).$$
\end{defn}
Note that when $H\equiv 0$ and $(\delta(0), \delta(1)) \in R$, we have $\mathcal A_H(\bar \gamma) = h(\delta(1)) - h(\delta(0))$,
which agrees with the definition in Section~\ref{sec:introduction}.
\begin{defn}
For $\bar u = (u,\ell) \in \widetilde{\mathcal M}_{{J}, {H}}(\bar\gamma_-, \bar\gamma_+;  {\alpha})$,
the energy is defined by $$E(u,\ell)=\int \left|\frac{\partial u}{\partial s}\right|^2ds\wedge dt=\mathcal A_{H}(\bar\gamma_-)-\mathcal A_{H}(\bar\gamma_+) - \sum_{j=1}^{|\Delta|}\mathcal A(\alpha(j)).$$
For $\bar u = (u,\ell) \in \widetilde{\mathcal M}_{\bold{J}, \bold{H}}(\bar\gamma_-, \bar\gamma_+;  \bold{\alpha})$,
the energy is defined by 
\begin{align*}
E(u,\ell) &=\int \left|\frac{\partial u}{\partial s}\right|^2ds\wedge dt \\
&=\mathcal A_{H_{-\infty}}(\bar\gamma_-)-\mathcal A_{H_{+\infty}}(\bar\gamma_+) - \sum_{j=1}^{|\Delta|}\mathcal A(\alpha(j)) - \int_Z \frac{\partial \bold H}{\partial s}(u)ds\wedge dt,
\end{align*}
where $H_{\pm \infty} := \lim_{s\to \pm \infty} H_s$.
\end{defn}
\subsection*{Index} 
\cbu
We define the index of the Hamiltonian chords, and the index of $J$-holomorphic curves.
\cb
\begin{defn}
We say $\bar \gamma = (\gamma, \delta) \in P_H$ is transverse, if $(D \phi^H_1)^{-1} D\iota \cdot T_{\delta(1)} L$ is transverse to $D\iota \cdot T_{\delta(0)}L$.
\end{defn}
\begin{defn}
For a transverse $\bar \gamma$, we define the index of $\bar\gamma$ by 
$$\ind\bar \gamma:= \mu (\{\Lambda_t\}_{t\in [0,1]}; \Lambda_0),$$ the Maslov index of the path of Lagrangian subspaces $\{\Lambda_t\}_{t\in [0,1]}$ with respect to a fixed Lagrangian $\Lambda_0$ inside $T_{\delta(0)}M$ as defined in \cite{rs-mip},
where $\Lambda_t = (D \phi^H_t)^{-1} \cdot \Lambda_t'$ and $\Lambda_t'$ continuous in $t\in [0,1]$ is a Lagrangian subspace of $T_{\gamma(t)} M$ such that
\begin{itemize}
\item $\Lambda_i' = D\iota (T_{\delta(i)}L)$ for $i \in \{0,1\}$ 
\item there exists $\Theta_t \in \R$ continuous in $t\in [0,1]$ such that  \cbu $\op{Det}^2(\Lambda_t') = e^{2\pi \sqrt{-1}\Theta_t}$ \cb and $\Theta_i =  \theta (\delta(i))$ for $i \in \{0,1\}$.
\end{itemize}
\end{defn}

In the case that $H\equiv 0$, $\gamma$ is trivial, and we require that  $(\delta(0), \delta(1)) \in R$. Then one can check that $\ind\bar \gamma = \ind(\delta(0), \delta(1))$ as in Formula~\ref{index of p,q}.

\begin{defn}\label{def: index}
For  $\bar u  \in \widetilde{\mathcal M}_{{J}, {H}}(\bar\gamma_-, \bar\gamma_+;  \bold{\alpha})$
we define the index of $\bar u$ by  $$\ind\bar u = \ind \bar \gamma_- -\ind \bar\gamma_+ -\sum_{j=1}^{|\Delta|}\ind \bold{\alpha}(j)+|\Delta|;$$
For $\bar u \in \widetilde{\mathcal M}_J(\bar\gamma_-, \emptyset;\alpha)$
we define the index of $\bar u$ by $$\ind\bar u = \ind \bar \gamma_-  -\sum_{j=1}^{|\Delta|}\ind \bold{\alpha}(j)+|\Delta|;$$
For $\bar u \in \widetilde{\mathcal M}_J(\emptyset, \bar\gamma_+ ;\alpha)$
we define the index of $\bar u$ by $$\ind\bar u = -\ind \bar \gamma_+  -\sum_{j=1}^{|\Delta|}\ind \bold{\alpha}(j)+|\Delta| + n.$$
\end{defn}

\cbu
\begin{prop}
For all the three cases in Definition~\ref{def: index}, $\ind\bar u$ equals to the Fredholm index of $\bar u$. 
\end{prop}
\begin{proof}
This is a standard result, and in the graded case the formula is simpler. See for example, Proposition 3.1.0 in \cite{alston2018exact} for the first moduli space and Proposition 5.3 in \cite{2013arXiv1311.2327A} for the second and third moduli spaces.
\end{proof}
\cb

\subsection*{Transversality}
Below are the collection of some standard transversality results:
\begin{prop} [Section 5 in \cite{alston2018exact}] \label{prop: transversality of pearls}
Suppose that $H$ is either admissible or constantly $0$. 
There is a Baire set $\mathcal J^{\op{reg}}$ of compatible almost complex structures such that for any $J \in \mathcal J^{\op{reg}}$,
if $\gamma_- \neq \gamma_+$,
 the moduli space ${\mathcal M}_{{J}, {H}}(\bar\gamma_-, \bar\gamma_+;  \bold{\alpha})$ is transversely cut-out, and in particular, it is a smooth manifold of dimension $\ind \bar u - 1$.
\end{prop}

\begin{proof}[Sketch of proof]
We can achieve transversality by perturbing the almost complex structure $J$ as long as the holomorphic curves are somewhere injective, which is guaranteed by $\gamma_- \neq \gamma_+$.
Section 5 in \cite{alston2018exact} contains the proof of this proposition for the admissible $H$ case, but the same proof also works for the $H\equiv 0$ case. 
\end{proof}

\begin{prop} [Section 3 in \cite{biran4221quantum}] \label{prop: transversality of pearly trajectories}
There is a Baire set $\mathcal Z^{\op{reg}}$ of tuples of compatible almost complex structures and Riemannian metrics such that
for any $(J, g)\in \mathcal Z^{\op{reg}}$, and for any $\bar\gamma = (p,q) \in R$:

if $\{p, q\} \cap W^u(x) = \emptyset$, then the moduli space
$\mathcal M_{ J} (x, \bar\gamma; \alpha)$ is transversely cut off, and in particular, it is a smooth manifold of dimension 
\cbu
$$\op{vir dim}{\mathcal M}_J(\bar\gamma, x; \alpha) := \ind x - \ind \bar\gamma - \sum_{j\in \Delta} \ind(\alpha(j)) + |\Delta| -1;$$
\cb
if $\{p, q\} \cap W^s(x) = \emptyset$, then the moduli space
$\mathcal M_{ J} (\bar\gamma,x; \alpha)$ is transversely cut off, and in particular, it is a smooth manifold of dimension 
\cbu
$$\op{vir dim}{\mathcal M}_J(x,\bar\gamma; \alpha) := \ind \bar \gamma - \ind x - \sum_{j\in \Delta} \ind(\alpha(j)) + |\Delta| - 1.$$
\cb
\end{prop}

\begin{proof}[Sketch of proof]
The proof in Section 3 of \cite{biran4221quantum} is carried out for the embedded Lagrangian case, but it also works in the immersed case.
\end{proof}

\begin{prop} [Section 5 in \cite{alston2018exact}] \label{prop: transversality of pearls 2}
Suppose that $H_{- \infty} \equiv 0$ and $H_{+\infty}$ is admissible, or $H_{- \infty}$ is admissible and $H_{+\infty}\equiv 0$.
There exist a Baire set $\mathbfcal J^{\op{reg}}$  of $s$-dependent compatible almost complex structures such that for any $\bold J\in \mathbfcal J^{\op{reg}}$,
  the moduli space $\widetilde{\mathcal M}_{\bold{J}, \bold{H}}(\bar\gamma_-, \bar\gamma_+;  \bold{\alpha})$ is transversely cut-out, and in particular, a smooth manifold of dimension $$\ind\bar\gamma - \sum_{j\in \Delta} \ind(\alpha(j)) + |\Delta| - \ind x.$$
\end{prop}

\begin{proof}[Sketch of proof]
The proof in Section 5 of \cite{alston2018exact}, which is for the case that both $H_\infty$ and $H_\infty$ are admissible, also works for the current case.
\end{proof}

\subsection*{Disc bubbles} \label{subsection: disc bubble}
In this section, we derive some inequality of indexes, which will be used later to exclude certain disc bubbles.
\cbu
Because of our positivity assumption, to derive these inequalities we do not need any transversality result.
\cb
The case that $H$ is admissible is easier and is taken care of in \cite{alston2018exact}. 
Now we focus on the case that $H\equiv 0$. 
Consider the moduli space ${\mathcal M}_{{J}}(\bar \gamma_-, \bar\gamma_+; \alpha)$ that satisfies $\mathcal A(\alpha(j)) >0$ for $j\in \{1,...,|\Delta|\}$.  
Here the reason why we assume $\mathcal A(\alpha(j)) >0$ is that to compactify moduli spaces of holomorphic strips without boundary punctures, we need to add holomorphic strips with boundary punctures along which trees of holomorphic discs are attached. These trees of holomorphic discs are non-constant, and hence have positive energy. 
Suppose that 
${\mathcal M}_{{J}}(\bar \gamma_-, \bar\gamma_+; \alpha)$ is not empty.
Then for any $(\Delta, \bar u) \in {\mathcal M}_{{J}}(\bar \gamma_-, \bar\gamma_+; \alpha)$ we have 
\begin{claim}\label{index of gamma and index of u}
$\ind\bar\gamma_- - \ind\bar\gamma_+ \geq \ind\bar u + 2 |\Delta|$.
\end{claim}
\begin{proof}
This follows directly from the weak positivity condition. 
 $\ind \bar \gamma_{-} -\ind \bar\gamma_{+} = \ind\bar u +\sum_{j=1}^{|\Delta|}\ind {\alpha}(j) - |\Delta| \geq \ind\bar u + 2|\Delta| $.
\end{proof}
We will see that if the moduli space ${\mathcal M}_{{J}}(\bar \gamma_-, \bar\gamma_+; \alpha)$ is transversely cut out, then $\ind(\bar u) \geq 1$. 
When there is no constant strip (also called ghost strip) in the moduli space ${\mathcal M}_{{J}}(\bar \gamma_-, \bar\gamma_+; \alpha)$ for any compatible $J$, 
we can perturb the moduli space to achieve transversality by varying $J$.
Constant strips can only appear, if and only if $\gamma_- = \gamma_+$. 
(Recall that $\bar\gamma_\pm = (\gamma_\pm, \delta_\pm)$ and in this case $\gamma_\pm$ is a constant map and $(\delta_\pm(0), \delta_\pm(1)) \in R$, but we write $\bar\gamma_\pm \in R$ for simplicity.)
\begin{claim}\label{claim: constant strip index}
Suppose that $\gamma_{-} = \gamma_{+}$ and that ${\mathcal M}_{{J}}(\bar \gamma_-, \bar\gamma_+; \alpha)$ is not empty.
Then $$\ind\bar\gamma_{-} - \ind\bar\gamma_{+} \geq 2.$$
\end{claim}
\begin{proof}
We have the following two cases:  
\begin{itemize}
\item  $\bar\gamma_{-} = \bar\gamma_{+}$. For $\bar u= (u, \ell)$ to be stable, we must have $\Delta \neq \emptyset$.
Since $u$ restricted to each component of $\R \times \{0,1\}$ is constant, there is clear notion of branches for $\ell$.
Around each element in $\Delta$, $\ell$ has a branch jump. Since for each $(p,q) \in R$ we have $\mathcal A(p,q) = -\mathcal A(q,p)$, on each component of $\R \times \{0,1\}$
there can be at most one branch jump. This contradicts to  $\bar\gamma_{-} = \bar\gamma_{+}$.
\item $\bar\gamma_{-} \neq \bar\gamma_{+}$. Denote $\bar\gamma_{-} = (p,q) \in R$, and then $\bar\gamma_+ = (q,p)$.
By the strong positivity condition, we get $\ind(p,q) \geq \frac{n+2}{2}$. Hence  $\ind\bar\gamma_{-} - \ind\bar\gamma_{+} = \ind(p,q) - \ind(q,p) = 2\ind(p,q) - n \geq 2$. (Figure~\ref{fig: ghost} is an illustration of the case when $|\Delta| = 2$ with two additional disc bubbles attached along $\Delta$.)
\end{itemize}
\end{proof}

\begin{rmk}
This claim shows why we impose the strong positivity condition instead of only the weak one.
\end{rmk}

Now we switch to the moduli space ${\mathcal M}_J(\bar\gamma, x; \alpha)$ with $x\in \op{Crit}f$ and $\bar\gamma = (p,q) \in R$
that satisfies $\mathcal A(\alpha(j)) >0$ for each $j\in \{1,...,|\Delta|\}$.  
Its virtual dimension satisfies 
\begin{align*}
\op{vir dim}{\mathcal M}_J(\bar\gamma, x; \alpha) & = \ind\bar\gamma - \sum_{j\in \Delta} \ind(\alpha(j)) + |\Delta| - \ind x - 1 \\
&\leq \ind\bar\gamma  - 2|\Delta| - \ind x - 1,
\end{align*}
where $\ind x $ is the Morse index. 

Suppose that ${\mathcal M}_J(\bar\gamma, x; \alpha)$ is not empty. In the transverse case, we have $\ind\bar\gamma - \ind x \geq \op{vir dim}({\mathcal M}_J(\bar\gamma, x; \alpha)) + 2|\Delta| + 1\geq 2|\Delta| + 1.$
The standard transversality argument does not work when ${\mathcal M}_J(\bar\gamma, x; \alpha)$ contains constant maps,
which can only happen when $p$ or $q$ lies in $W^s(x)$ the stable manifold of $x$.
If $(f,g)$ is generic, this only happens when $\ind x = 0$.
\begin{claim}
Suppose that ${\mathcal M}_J(\bar\gamma, x; \alpha)$ is not empty,
where $x\in \op{Crit}f$ with $\ind x = 0$ and $\bar\gamma  \in R$.
Suppose also that  $\mathcal A(\alpha(j)) >0$, 
 for each $j\in \{1,...,|\Delta|\}$  and $\Delta \neq \emptyset$. Then $\ind \bar\gamma - \ind x \geq 3$.
\end{claim}
\begin{proof}
For any $\bar u \in {\mathcal M}_J(\bar\gamma, \emptyset; \alpha)$,  one has $E(\bar u) \geq 0$. Therefore, $\mathcal A(\bar\gamma) > 0$, and hence $\ind \bar\gamma - \ind x = \ind\bar\gamma \geq 3$.  
\end{proof}

Similarly,
\begin{claim}\label{claim: empty to gamma}
Suppose that ${\mathcal M}_J(x,\bar\gamma; \alpha)$ is not empty,
where $x\in \op{Crit}f$ with $\ind x = n$ and $\bar\gamma  \in R$.
Suppose also that  $\mathcal A(\alpha(j)) >0$, 
 for each $j\in \{1,...,|\Delta|\}$  and $\Delta \neq \emptyset$. 
 \cbu 
 Then $\ind x - \ind \bar\gamma \geq 3$.
 \cb
\end{claim}

\subsection*{Degeneration at an embedded point of $\iota$} 
When we compactify the moduli space $\mathcal M_{J}(\bar\gamma_-, \bar\gamma_+)$, $\mathcal M_{J}(\bar\gamma, x)$ or $\mathcal M_{J}(x, \bar\gamma)$, there is another bad degeneration that we want to rule out (recall that these moduli spaces refer to $H\equiv 0$ because $H$ is not included in the notation).
Namely, the moduli space can break at an embedded point of $\iota$. By the exactness of $\iota$, such degenerations come in pairs as below: $$(\bar u_-, \bar u_+) \in \mathcal M_{J}(\bar\gamma_-, \emptyset; \alpha_-) {}_{\op{ev}_{+\infty}}\times_{\op{ev}_{-\infty}} \mathcal M_{J}(\emptyset, \bar\gamma_+; \alpha_+),$$
where $\bar\gamma_\pm \in R$, and $\mathcal A(\alpha_\pm(j_\pm)) >0$, for any $j_\pm \in \{1,...,|\Delta_\pm|\}$. (Here we allow $\Delta_\pm = \emptyset$).

\begin{claim}\label{RxR}
If the moduli space $\mathcal M_{J}(\bar\gamma_-, \emptyset; \alpha_-) {}_{\op{ev}_{+\infty}}\times_{\op{ev}_{-\infty}} \mathcal M_{J}(\emptyset, \bar\gamma_+; \alpha_+)$ that satisfies the above condition is not empty, then 
$$\ind\bar\gamma_- - \ind\bar\gamma_+ \geq 2.$$
\end{claim}
\begin{proof}
This directly follows from the strong positivity condition.  Firstly, $\ind\bar\gamma_- \geq \frac{n+2}{2}$. Denote $\bar\gamma_+ =:(p,q)$, and then $\ind(q,p) \geq \frac{n+2}{2}$. Therefore, $\ind\bar\gamma_- - \ind\bar\gamma_+ \geq \frac{n+2}{2} - (n - \ind(q,p))\geq 2.$
\end{proof}

\section{Lagrangian Floer cohomology via pearly trajectories} \label{lagrangian floer via pearly trajectories}

We define the Floer cochain complex $CF^{*}(\iota)=\Z_2 \text{Crit}f \oplus\mathbb{Z}_{2}R=: \bold C\oplus \bold R,$
where $\Z_2 \text{Crit}f$ is the free $\mathbb{Z}_{2}$-module
generated by $\text{Crit}f$,
and $\mathbb{Z}_{2}R$ is the free $\mathbb{Z}_{2}$-module generated
by $R.$

We give $CF^{*}(L)$ a $\mathbb{Z}$-grading using $\ind$ and define the differential $d:CF^{*}(\iota)\to CF^{*+1}(\iota)$
by 
\[
d=\left(\begin{array}{cc}
d_{\bold C\bold C} & d_{\bold C\bold R}\\
d_{\bold R\bold C} & d_{\bold R\bold R}
\end{array}\right),
\]
 where 
\begin{itemize}
\item $d_{\bold C\bold C} : \bold C \to \bold C$ is the standard Morse differential, i.e.,
\[
d_{\bold C \bold C}x =\sum_{x'\in\text{Crit}f,\ind x'=\ind x +1}\sharp \mathcal M (x', x)\cdot x',
\]
where $\mathcal M(x', x)$ is the space of Morse trajectories from $x'$ to $x$ mod the translation of the domain. More precisely, \begin{align*}\mathcal M(x', x) =  \{ & u: \R \to M ~|~  \dot u = -\nabla f (u), \\
  & \lim_{s\to -\infty} u(s) = x' \text{ and } \lim_{s\to \infty} u(s) = x \}. 
  \end{align*} Here $\nabla f$ is the gradient of $f$ with respect to the metric $g$;
\item $d_{\bold C\bold R}: \bold R \to \bold C$ is defined over the generators by 
\[
d_{\bold C \bold R}\bar\gamma=\sum_{x\in\text{Crit}f,\ind x=\ind \bar\gamma+1}\sharp \mathcal M_{ J} (x, \bar\gamma)\cdot x.
\]
 See Figure~\ref{fig:four types of trajectories}(b) ($\bar \gamma = (p,q)$);
\item $d_{\bold R\bold C}: \bold C \to \bold R$ is defined over the generators by 
\[
d_{\bold R\bold C } x =\sum_{\bar\gamma \in R,\ind \bar\gamma=\ind x+1}\sharp \mathcal M_{ J} (\bar\gamma, x)\cdot \bar\gamma.
\]
 See Figure~\ref{fig:four types of trajectories}(c) ($\bar \gamma = (p,q)$);
\item $d_{\bold R\bold R}: \bold R \to \bold R$ is defined over the generators by 
\[
d_{\bold R\bold R } \bar \gamma =\sum_{\bar\gamma' \in R,\ind \bar\gamma' =\ind \bar\gamma +1}\sharp \mathcal M_{ J} (\bar\gamma', \bar\gamma)\cdot \bar\gamma'.
\]
 See Figure~\ref{fig:four types of trajectories}(d) ($\bar\gamma' = (r,s)$ and $\bar \gamma = (p,q)$).
\end{itemize}


\begin{rmk}
\label{lem:Under-strong-positivity}
In general, $d$ is supposed to count all the pearly trajectories,
and in our case, the exactness of $\iota$ rules out trajectories
that contain smooth discs. 
One might want to include in the definition of $d_{\bold R \bold R}$ a pearly trajectory that consists of an element in $ \mathcal M_J(\bar\gamma', \emptyset)$  and an element in $\mathcal M_J(\emptyset,\bar\gamma)$ connected by a Morse gradient trajectory over a finite time interval pictured in Figure~\ref{fig:four types of trajectories}(e) ($\bar\gamma' = (r,s)$ and $\bar \gamma = (p,q)$).
Under the strong positivity condition, any such pearly trajectory has
$\ind \bar\gamma' - \ind \bar \gamma \geq 2$ and hence does not contribute to the differential. This follows from a similar argument as the proof of Claim~\ref{RxR}.
\end{rmk}

\begin{lemma}
\label{prop:d-is-defined.} 
There is a Baire set $\mathcal Z^{\op{reg}}$ of tuples of compatible almost complex structures and Riemannian metrics such that
for any $(J, g)\in \mathcal Z^{\op{reg}}$,
$d$ is well-defined.
\end{lemma}
\begin{proof}
By Proposition~\ref{prop: transversality of pearls} and Proposition~\ref{prop: transversality of pearly trajectories} the moduli spaces involved in the definition of $d$ are all $0$-dimensional manifolds.
Now we only need to show that they are compact.

We show that $\mathcal M_J(x,\bar\gamma)$ is compact, and leave the rest for readers.
Given any sequence of trajectories in $\mathcal M_J(x,\bar\gamma)$, the Gromov's companctness and the exactness of $\iota$ implies \cbu there exists a subsequence converging \cb to a broken trajectory  $$\bold u = ( u_{1}, ..., u_{k-1},([\bar u_{k}],\mathbf{v}_{k}), ([\bar u_{k+1}],\mathbf{v}_{k+1}),...,([\bar u_{k+m}],\mathbf{v}_{k+m}))$$ of length $k+m$ 
where 
\begin{enumerate}
    \item $u_i$ is a Morse trajectory from  $x_i$ to $x_{i+1}$  for  $i \in \{1,...,k-1\}$, $x_i \in \op{Crit}f$  for $i \in \{1,...,k\}$, and  $x_1 = x$;
    \item $[\bar u_{k}] \in {\mathcal{M}}_J(x_k,\bar\gamma_{k+1};\bold{\alpha}_{i})$, where $\bar\gamma_{k+1} \in R$;
    \item $[\bar u_{i}]$ is either an element in ${\mathcal{M}}_J(\bar\gamma_{i},\bar\gamma_{i+1};\bold{\alpha}_{i})$,
    or a pair in $$\mathcal M_{J}(\bar\gamma_i, \emptyset; \alpha_i) {}_{\op{ev}_{+\infty}}\times_{\op{ev}_{-\infty}} \mathcal M_{J}(\emptyset, \bar\gamma_{i+1}; \alpha'_{i}),$$  
    with  $\bar \gamma_i \in R$, for $i \in \{k+1,...,k+m\}$, and $\bar\gamma_{k+m+1}=\bar\gamma$; and 
    \item  $\mathbf v_i$ is a possibly empty set of holomorphic trees attached to $[\bar u_i]$  along $\Delta_i$  for $i \in \{k,...,k+m\}$.
\end{enumerate}
\cbu
Note that we are taking a standard Gromov compactification, so the limit of the $J$-holomorphic part does not contain any Morse trajectory between $\op{ev}_{+\infty}$ and $\op{ev}_{-\infty}$ in (3). \cb
Since $u_i$ is regular for $i\in \{1,...,k-1\}$, $$\ind x_i - \ind x_{i+1} \geq 1.$$
For $i \in \{k+1, ..., k+m\},$
\begin{itemize}
\item if $[\bar u_{i}] \in {\mathcal{M}}_J(\bar\gamma_{i},\bar\gamma_{i+1};\bold{\alpha}_{i})$ 
\begin{itemize}
\item if $\gamma_{i} \neq \gamma_{i+1}$, then by Proposition~\ref{prop: transversality of pearls} and Claim~\ref{index of gamma and index of u}
we have $$\ind \bar\gamma_i - \ind \bar\gamma_{i+1}\geq \ind \bar u_i + 2|\Delta_i|\geq 1+ 2|\Delta_i|;$$
 \item if $\gamma_i = \gamma_{i+1}$, then by Claim~\ref{claim: constant strip index} $$\ind\bar\gamma_i - \ind\bar\gamma_{i+1} \geq 2.$$
 \end{itemize}
\item if $[\bar u_{i}] \in \mathcal M_{J}(\bar\gamma_i, \emptyset; \alpha_i) {}_{\op{ev}_{+\infty}}\times_{\op{ev}_{-\infty}} \mathcal M_{J}(\emptyset, \bar\gamma_{i+1}; \alpha'_{i}),$ by Claim~\ref{RxR} $$\ind\bar\gamma_i - \ind\bar\gamma_{i+1} \geq 2.$$
 \end{itemize}
 For the piece $\bar u_k$, denote $\bar\gamma_{k+1} =(p,q)$. 
 \begin{itemize}
 \item If $\{p,q\} \cap  W^u(x_k) = \emptyset$, then by Proposition~\ref{prop: transversality of pearly trajectories} one gets $$\ind x_k - \ind\bar\gamma_{k+1}  \geq \op{vir dim}({\mathcal M}_J(x_k, \bar\gamma_{k+1}; \alpha_k)) + 2|\Delta| + 1\geq 2|\Delta| + 1.$$ 
 \item If $\{p,q\} \cap  W^u(x_k) \neq \emptyset$, by Claim~\ref{claim: empty to gamma} one obtains $$\ind x_k - \ind \bar\gamma_{k+1} \geq 3.$$
\end{itemize}
In summary, 
\begin{align*}
1& =  \ind x - \ind\bar\gamma \\
& = \sum_{i=1}^{k-1} (\ind x_i - \ind x_{i+1}) + (\ind x_{k} - \ind \bar\gamma_{k+1}) + \sum_{i=k+1}^{k+m} (\ind \bar\gamma_i -\ind \bar\gamma_{i+1}) \\
&\geq  k-1 + \min(3, 1+ 2|\Delta_k|)+ \sum_{i=k+1}^{k+m} \min(2, 1+ 2|\Delta_i|).
\end{align*}
We conclude $k = 1$, $m = 0$ and $\Delta_1= \emptyset$. Hence, $\bold u \in \mathcal M_J(x,\bar\gamma)$. 
\end{proof}

\begin{prop}\label{d square equals 0}
Under the same condition as Proposition~\ref{prop:d-is-defined.}, 
$d^{2}=0$.
\end{prop}
\begin{proof}[Sketch of proof]
The proof is a standard argument by studying the boundary of certain $1$-dim moduli spaces.
Note that
\[
d^{2}=\left(\begin{array}{cc}
d_{\bold{CC}}^{2}+d_{\bold{CR}}d_{\bold{RC}} & d_{\bold{CC}}d_{\bold{CR}}+d_{\bold{CR}}d_{\bold{RR}}\\
d_{\bold{RC}}d_{\bold{CC}}+d_{\bold{RR}}d_{\bold{RC}} & d_{\bold{RC}}d_{\bold{CR}}+d_{\bold{RR}}^{2}
\end{array}\right).
\]
For example, to show that $d_{\bold{RC}}d_{\bold{CC}}+d_{\bold{RR}}d_{\bold{RC}} = 0$, we look at the boundary (in the sense of Gromov's compactification) of the moduli space $\mathcal M_J(x, \bar\gamma)$ with $x\in \op{Crit}f$, $\bar\gamma\in R$ and $\ind x - \ind \bar\gamma = 2$.
Using a similar argument as in the proof of Proposition~\ref{prop:d-is-defined.} to rule out ``bad'' degenerations and standard gluing results (Section 4 in \cite{piunikhin1996symplectic}), 
we show that all the boundary components of $\mathcal M_J(\bar\gamma, x)$ are given by gluing 
$\mathcal M_J(\bar\gamma, x')$ and $\mathcal M(x', x)$ for all $x'\in \op{Crit}f$ with $\ind x' = \ind x + 1$, 
or gluing
$\mathcal M_J(\bar\gamma, \bar\gamma')$ and $\mathcal M_J(\bar\gamma', x)$ for all $\bar\gamma' \in R$ with $\ind \bar \gamma' = \ind\bar\gamma + 1$. See Figure~\ref{fig: dMR}.
\end{proof}

\begin{figure} 
\centering
\begin{tikzpicture}[scale = 0.7, every node/.style={scale=0.6}]
\draw[->] (2,1) -- (2.5,1);
\draw (2.5,1) -- (3,1);
\filldraw (3,1) circle [radius = 0.05];
\draw[->] (3,1) -- (4,1);
\draw (4,1) -- (5,1);
\node[above] at (3,1.2) {$x'$};
\node[right] at (5.1, 1) {$x$};
\draw 
(0, 1) .. controls (0.5, 1.5) and (2,2) .. (2, 1)
(0, 1) .. controls (0.5,0.5) and (2, 0) .. (2,1);
\node[left] at (0,1) {$\bar\gamma$};
\node[right] at (0,1) {$-$};
\draw (0, 1) circle [radius = 0.05];

\draw (6, 1) .. controls (6.5, 1.7) and (8.5,1.7) .. (9, 1)
(6,1) .. controls (6.5,0.3) and (8.5, 0.3) .. (9,1);
\node[above] at (9,1.2){$\bar\gamma'$};
\node[left] at (6,1) {$\bar\gamma$};
\node[right] at (6,1) {$-$};
\node[left] at (9,1) {$+$};
\node[right] at (9,1) {$-$};

\draw (6,1) circle [radius = 0.05];
\draw[->] (11,1) -- (11.5,1);
\draw (11.5,1) -- (12,1);
\draw (9,1) circle [radius = 0.05];
\node[right] at (12.2,1) {$x$};
\filldraw (12,1) circle [radius = 0.05];
\draw 
(9, 1) .. controls (9.5, 1.5) and (11,2) .. (11, 1)
(9, 1) .. controls (9.5, 0.5) and (11, 0) .. (11, 1);
\node[right] at (0,1) {$-$};
\draw (0, 1) circle [radius = 0.05];
\filldraw (5,1) circle [radius = 0.05];

\end{tikzpicture}
\caption[]{$\partial \mathcal M_J(\bar\gamma, x)$}
\label{fig: dMR}
\end{figure}
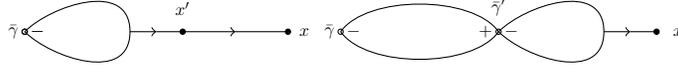

%
%
%
%
%

We define the Lagrangian Floer cohomology of $\iota$ by $HF^{*}(\iota) = \ker d/\op{im}d $.

\begin{proof}[Proof of Corollary~\ref{rank inequality}]
Denote $\bold R^{+} := \Z_2 \{ (p,q) \in  R ~|~ \mathcal A(p,q) \geq  0 \}$ and 
$\bold R^{-} :=\Z_2 \{(p,q) \in R ~|~ \mathcal A(p,q) < 0\}$, and then $\bold R = \bold R^{+} \oplus \bold R^{-}$. 
\cbu
We give the complex $\mathfrak C = CF^*(\iota)$ an filtration 
\[\mathfrak C = \mathcal F^0 \mathfrak C \supseteq \mathcal F^1 \mathfrak C \supseteq  \mathcal F^2 \mathfrak C \supseteq 0\]
by 
$\mathcal F^1 \mathfrak C = \bold C \oplus \bold{R}^{-}$,
$ \mathcal F^2 \mathfrak C  = \bold R^- $.
Then we denote by $E^{a,b}_r$ the spectral sequence associated to $\mathcal F^* \mathfrak C$. 
In particular, $E_0^{0,b} = (\bold R^{+})^b$, $E_0^{1,b} = \bold C^{1+b}$, and $E_0^{2,b} = (\bold R^-)^{2+b}$.
At the $E_1$ page, $E_1^{a,b} = 0$ for $a \notin \{0,1,2\}$,
and for $a\in \{0,1,2\}$, it is given by:
\[
 E_1^{0,b} = H^b(\bold R^+) \overset{d_1^{0,b}}{\to} E_1^{1,b} = H^{1+b}(\bold C) \overset{d_1^{1,b}}{\to} E_1^{2,b} = H^{2+b}(\bold R^-),
\]
where the cohomologies are calculated using the induced differential by $d$ on the homogeneous summands.
Then we have 
\begin{align*}
\op{rank} E_2^{1,b} \geq & \op{rank} H^{1+b}(\bold C) - \op{rank} H^b(\bold R^+) - \op{rank} H^{2+b}(\bold R^-) \\
= &  \op{rank} H^{1+b}(L;\Z_2) - \op{rank} H^b(\bold R^+) - \op{rank} H^{2+b}(\bold R^-).
\end{align*}
On the other hand, since the spectral sequence converges to $HF^*(\iota)$ and $E_2^{1,b} = E_{\infty}^{1,b}$, we have $$\op{rank} E_2^{1,b} \leq \op{rank} E_{\infty}^{0, b+1} + \op{rank} E_{\infty}^{1,b} + \op{rank} E_{\infty}^{2, b-1} =  \op{rank} HF^{1+b} (\iota).$$
Now putting these together and summing over $b$ gives us 
\begin{align*}
& \sum_b \op{rank} HF^{1+b} (\iota) \\
 \geq & \sum_b  \op{rank} H^{1+b}(L;\Z_2) - \sum_b \op{rank} H^b(\bold R^+) - \sum_b \op{rank} H^{2+b}(\bold R^-) \\
  \geq  & \sum_b  \op{rank} H^{1+b}(L;\Z_2) - |R|
\end{align*}
\cb
\end{proof}

\section{Lagrangian Floer cohomology via Hamiltonian perturbation}\label{lagrangian floer via hamiltonian perturbation}
In this section, we recall $HF^*_H(\iota)$ an alternative definition of Lagrangian Floer cohomology using Hamiltonian perturbation, 
which is an invariant of $\iota$ under Hamiltonian deformation, 
and show that $HF^*(\iota)$ is isomorphic to $HF^*_H(\iota).$
As a consequence, $HF^*(\iota)$ is independent of choices and is invariant under Hamiltonian deformation.

We define the Hamiltonian perturbed Lagrangian Floer cochain complex by $CF_H^* (\iota) = (\Z_2 P_H, d_H)$,
where $\Z_2 P_H$ is the free $\Z_2$-module generated by Hamiltonian paths $P_H$ and the differential $d_H$ is defined by 
\[
d_H \bar\gamma = \sum_{\ind \bar\gamma' = \ind \bar\gamma + 1} \sharp \mathcal M_{J,H}(\bar\gamma', \bar\gamma)\cdot \bar\gamma'.
\]
Then we define $HF^*_H(\iota) = \ker d_H / \op{im}d_H$. 
\begin{thm}[\cite{alston2018exact}] \label{thm: Hamiltonian perturbed Floer}
For an admissible $H$, there exists a Baire set $\mathcal J^{\op{reg}}_H$ of compatible almost complex structures, such that for any $J\in \mathcal J^{\op{reg}}_H$ the Hamiltonian perturbed Lagrangian Floer cohomology $HF_H^* (\iota)$ is well-defined, independent of the choice of $H$ and $J$, and invariant under Hamiltonian deformation.
\end{thm}

The isomorphism between $HF^*(\iota)$ and $HF^*_H(\iota)$ is a type of PSS isomorphism (see \cite{piunikhin1996symplectic}).
Let $(J,g)\in \mathcal Z^{\op{reg}}$ as in Proposition~\ref{prop:d-is-defined.},
and  $(H, J')\in \mathcal J^{\op{reg}}_H $ as in Theorem~\ref{thm: Hamiltonian perturbed Floer}.
Let $(\bold H, \bold J)$ be a smooth interpolation from $(0, J)$ to $(H,J')$.
To show that $HF^*(\iota)$ and $HF^*_H(\iota)$ are isomorphic, we construct chain maps between them, and show that the composition is chain homotopic to an isomorphism. 
Denote $\bold \Gamma := CF_H^* (\iota)$ and $CF^*(\iota) = \bold C \oplus \bold R$.
Then we define $\Phi : CF_H^* (\iota) \to CF^*(\iota)$ by defining $\Phi_{\bold C}: \bold \Gamma\to \bold C$ and $\Phi_{\bold R}:\bold \Gamma\to \bold R$ as follows.

For $\bar\gamma\in P_H$, we define 
\begin{enumerate}
\item $$
\Phi_{\bold C}(\bar\gamma)=\sum_{x\in \op{Crit}f, \ind x = \ind \bar\gamma} \sharp \widetilde{\mathcal M}_{\bold J, \bold H}(x, \bar\gamma) \cdot x.
$$
See Figure~\ref{fig: three trajectories} (I) and formula \eqref{eq:moduli space for chain map case I}.
\item  $$
\Phi_{\bold R,1}(\bar\gamma)=\sum_{\bar\gamma' \in R, \ind \bar\gamma' = \ind \bar\gamma} \sharp \widetilde{\mathcal M}_{\bold J, \bold H}(\bar\gamma', \bar\gamma) \cdot \bar\gamma'.
$$
See Figure~\ref{fig: three trajectories} (II) with $\bar\gamma' = (p,q)$ and formula \eqref{eq:moduli space for chain map case II}.
\item $$
\Phi_{\bold R,2}(\bar\gamma)=\sum_{\bar\gamma' \in R, \ind \bar\gamma' = \ind \bar\gamma} \sharp \widetilde{\mathcal M}^{\text{pearl}}_{\bold J, \bold H}(\bar\gamma', \bar\gamma) \cdot \bar\gamma'.
$$
See Figure~\ref{fig: three trajectories} (III) with $\bar\gamma' = (p,q)$ and formula \eqref{eq:moduli space for chain map case III}.
\item $\Phi_{\bold R} = \Phi_{\bold R,1} + \Phi_{\bold R, 2}$.
\end{enumerate}
We similarly define $\Psi : CF^*(\iota) \to CF_H^* (\iota)$ by counting the moduli spaces in formulas \eqref{eq:moduli space for chain map case I opposite}, \eqref{eq:moduli space for chain map case II opposite}, and \eqref{eq:moduli space for chain map case III opposite}.
Note that $\Phi$ counts trajectories from generators of $CF^*(\iota)$ to generators of $CF^*_H(\iota)$ while $\Psi$ counts trajectories  from generators of $CF^*_H(\iota)$ to generators of $CF^*(\iota)$.

\begin{prop}\label{prop: chain map}
There exists a Baire set $\mathbfcal{Z}^{\op{reg}}$ of pairs of $(\bold H, \bold J)$ connecting $(0, J)$ to $(H,J')$, such that 
for any $(\bold H, \bold J)\in \mathbfcal{Z}^{\op{reg}}$, $\Phi$ is well-defined and satisfies $\Phi d_H=d\Phi.$
Moreover, the chain map $\Phi'$ defined using a different choice of $(\bold H', \bold J')\in \mathbfcal{Z}^{\op{reg}}$ is chain homotopic to $\Phi$.

 \end{prop}
\begin{proof}[Sketch of proof]
A similar argument as the proof of Proposition \ref{prop:d-is-defined.}
shows that $\Phi$ is well-defined. 

The chain map equation $(\Phi d_H + d\Phi)\bar\gamma =0$ can be written as 
\begin{align*}
(\Phi_{\bold C}+\Phi_{\bold R, 1} + \Phi_{\bold R, 2})d_H\bar\gamma + (d_{\bold{CC}} + d_{\bold {RC}})\Phi_{\bold C}\bar\gamma &\\ 
+(d_{\bold{CR}} +d_{\bold{RR}})\Phi_{\bold R,1} \bar\gamma + (d_{\bold{CR}} +d_{\bold{RR}})\Phi_{\bold R,2} \bar\gamma  & = 0.
\end{align*}
First, for any $\bar\gamma'$ in the definition of $\Phi_{\bold R,2}$, we have $\mathcal A(\bar\gamma') > 0$, and hence  $$ d_{\bold{CR}} \Phi_{\bold R, 2}\bar\gamma = 0.$$
It is enough to show that for any $x\in \text{Crit}f$ with $\ind x = \ind \bar\gamma + 1$, we have 
\begin{equation} \label{cat}
\langle (\Phi_C d_H + d_{\bold{CC}}\Phi_C + d_{\bold{CR}}\Phi_{\bold R,1})\bar\gamma, x\rangle = 0
\end{equation}
and for any $\bar\gamma'' \in R$ with $\ind\bar\gamma'' = \ind \bar\gamma + 1$,
$$\langle ((\Phi_{\bold R,1} + \Phi_{\bold R,2})d_H + d_{\bold{RC}}\Phi_{\bold C} + d_{\bold{RR}}(\Phi_{\bold R,1} + \Phi_{\bold R,2}))\bar\gamma,\bar\gamma''\rangle = 0.$$
We look at the boundaries of the $1$-dimensional moduli spaces $\widetilde{\mathcal M}_{\bold J, \bold H}(x, \bar\gamma)$,  $\widetilde{\mathcal M}_{\bold J, \bold H}(\bar\gamma'', \bar\gamma)$, and  $\widetilde{\mathcal M}^{\text{pearl}}_{\bold J, \bold H}(\bar\gamma'', \bar\gamma)$. 
A similar index calculation using positivity assumption as in the proof of Proposition ~\ref{prop:d-is-defined.}  and standard gluing results show that counting $\partial \widetilde{\mathcal M}_{\bold J, \bold H}(x, \bar\gamma)$ gives  Equation~\ref{cat}.
Similarly, $\partial \widetilde{\mathcal M}_{\bold J, \bold H}(\bar\gamma'', \bar\gamma)$ corresponds to 
$$\langle \Phi_{\bold R, 1}d_H \bar\gamma + d_{\bold{RR}}\Phi_{\bold R,1}\bar\gamma, \bar\gamma'' \rangle+ n_{\bar\gamma,\bar\gamma''}, $$
where $n_{\bar\gamma,\bar\gamma''} := \sharp  \mathcal M_J(\bar\gamma'', \emptyset)_{\op{ev}_{+\infty}}  \times_{\op{ev}_{-\infty}} \widetilde{\mathcal M}_{\bold J, \bold H}(\emptyset, \bar\gamma)$.
Finally, using Remark~\ref{lem:Under-strong-positivity} again we see that $\partial \widetilde{\mathcal M}^{\text{pearl}}_{\bold J, \bold H}(\bar\gamma', \bar\gamma) $ corresponds to $$\langle \Phi_{\bold R,2}d_H\bar\gamma + d_{\bold{RC}}\Phi_{\bold C} \bar\gamma + d_{\bold{RR}}\Phi_{\bold R, 2}\bar\gamma, \bar\gamma'' \rangle + n_{\bar\gamma,\bar\gamma''}.$$
The statement that $\Phi$ and $\Phi'$ are chain homotopic follows from a standard argument by considering the boundary of certain moduli spaces arising from a homotopy from $(\bold J, \bold H)$ to $(\bold J', \bold H')$.
We leave this to the reader.
\end{proof}

\begin{thm}\label{thm: iso}
The maps $\Phi$ and $\Psi$ induce isomorphisms between $HF_H^*(\iota)$ and $HF^*(\iota)$. In particular, $HF^*(\iota)$ is  an invariant of $\iota$ under Hamiltonian deformation.
\end{thm}
\begin{proof}[Sketch of proof]
\cbu 
The plan is to show that $\Phi \circ \Psi $ is homotopic to the identity map $\op{Id}_{CF^* (\iota)}$
and $\Psi \circ \Phi$ is homotopic to the identity map $\op{Id}_{CF^*_H (\iota)}$.

Let $(\bold H, \bold J)$ be a homotopy from $(0, J)$ to $(H, J')$ as in Proposition~\ref{prop: chain map}, 
such that
\be
\item $\bold H = 0$ if $s \leq -1$ and $\bold H = H$ if $s \geq 1$,
\item $\bold J = J$ if $s \leq -1$ and $\bold J = J'$ if $s \geq 1$;
\ee
and $(\bar{\bold H}, \bar{\bold J})$ be a homotopy from $(H,J')$ to $(0,J)$
such that 
\be
\item $\bar{\bold H} = H$ if $s \leq -1$ and $\bar{\bold H} = 0$ if $s \geq 1$,
\item $\bar{\bold J} = J'$ if $s \leq -1$ and $\bar{\bold J} = J$ if $s \geq 1$.
\ee 
\n \\
\noindent
{\bf Step 1. $\Phi \circ \Psi $ is homotopic to $\op{Id}_{CF^* (\iota)}$.} \cb
Intuitively we achieve this by gluing the moduli spaces in the definition of $\Psi$ and $\Phi$,
and then homotoping the Hamiltonian to $0$, and the $s$-dependent almost complex structure to an $s$-independent one.

Similar as the definition of $\Phi$ (and also $\Phi_{\bold C}, \Phi_{\bold R,1},$ and $\Phi_{\bold R, 2}$), we have corresponding maps in the reverse directions: 
\begin{align*}
\Psi_{\bold C}:& \bold C \to  CF_H^* (\iota), &  \Psi_{\bold R,1}: & \bold R \to  CF_H^* (\iota), & \Psi_{\bold R,2}:&\bold R \to CF_H^*(\iota),
\end{align*}
\begin{align*}
 \Psi_{\bold R}  = & \Psi_{\bold R,1}+\Psi_{\bold R,2}, & 
 \Psi = & \Psi_{\bold R}  \oplus \Psi_{\bold C} : \bold R \oplus \bold C \to CF_H^*(\iota). 
\end{align*}

\cbu 

We will define a chain homotopy $K: CF^*(\iota) \to CF^*(\iota)$ and show that 
\begin{equation} \label{eqn: PHP}
    \Phi \circ \Psi - \op{Id}_{CF^*(\iota)} = [K, d].
\end{equation}
We choose a glued Hamiltonian by $\widetilde{\mathbb H}_\lambda:= \bold H \sharp_\lambda \bar{\bold H}$ such that 
\be 
\item for $\lambda \geq 1$,  
    $\widetilde{\mathbb H}_\lambda (s) =
    \begin{cases}
      \bold H(s + \lambda) & \text{if $s \leq 0$ }\\
      \bar{\bold H}(s - \lambda) & \text{if $s \geq 0$ }
    \end{cases}$,
\item for $\lambda \leq -1$,  $\widetilde{\mathbb H}_\lambda = 0$,
\item $\{\widetilde{\mathbb H}_\lambda\}_{\lambda \in [-1, 1]} $ is a smooth homotopy from $\widetilde{\mathbb H}_{-1}$ to $\widetilde{\mathbb H}_1$. 
\ee 
Let $\mathbb H_\lambda$ be a generic $C^\infty$-perturbation of $\widetilde{\mathbb H}_\lambda$ such that the perturbation goes to $0$ as $\lambda$ or $s$ goes to $\infty.$
Similarly, we define a glued almost complex structure by $\widetilde{\mathbb J}_\lambda:= \bold J \sharp_\lambda \bar{\bold J}$ by 
\be 
\item for $\lambda \geq 1$,  
    $\widetilde{\mathbb J}_\lambda (s) =
    \begin{cases}
      \bold J(s + \lambda) & \text{if $s \leq 0$ }\\
      \bar{\bold J}(s - \lambda) & \text{if $s \geq 0$ }
    \end{cases}$,
\item for $\lambda \leq -1$, $\widetilde{\mathbb J}_\lambda = J,$
\item $\{\widetilde{\mathbb J}_\lambda\}_{\lambda \in [-1, 1]} $ is a smooth homotopy from $\widetilde{\mathbb J}_{-1}$ to $\widetilde{\mathbb J}_1$. 
\ee
Let $\mathbb J_\lambda$ be a generic $C^\infty$-perturbation of $\widetilde{\mathbb J}_\lambda$ such that the perturbation goes to $0$ as $\lambda$ or $s$ goes to $\pm \infty.$

We now define the following moduli spaces:

\begin{enumerate}[label = (M\arabic*)]

\item For any $x,x'\in \text{Crit}f$, we define the moduli space $$\widetilde{\mathcal M} (x, x') = \coprod_{\lambda} \widetilde{\mathcal M} _\lambda (x, x')$$
where $\widetilde{\mathcal M} _\lambda (x, x')$ is the moduli space of $\mathbb H_\lambda$-perturbed $\mathbb J_\lambda$-holomorphic strips $(u,\ell)$ with removable singularities at $\pm \infty$ such that $\lim_{s\to -\infty} \ell (s, 0) = \lim_{s\to -\infty} \ell (s, 1) \in  W^u(x)$ and $\lim_{s\to +\infty} \ell (s,0) = \lim_{s \to +\infty} \ell (s,1) \in  W^s(x')$.  

\item \label{M2} For any $\bar\gamma \in R$ and $x\in \text{Crit}f$, we define the moduli space $$\widetilde{\mathcal M}_1 (\bar\gamma, x) = \coprod_{\lambda} \widetilde{\mathcal M} _{1,\lambda} (\bar\gamma, x)$$
where $\widetilde{\mathcal M} _{1,\lambda} (\bar\gamma, x)$ is the moduli space of $\mathbb H_\lambda$-perturbed $\mathbb J_\lambda$-holomorphic strips $(u,\ell)$ with removable singularities at $+ \infty$ such that  $\lim_{s\to +\infty} \ell (s,0) = \lim_{s \to +\infty} \ell (s,1) \in  W^s(x)$ and $\lim_{s\to -\infty} (u(s,\cdot),\ell(s,\cdot)) = \bar\gamma.$

\item For any $\bar\gamma \in R$ and $x\in \text{Crit}f$, we define the moduli space $$\widetilde{\mathcal M}_1 (x, \bar\gamma) = \coprod_{\lambda} \widetilde{\mathcal M} _{1, \lambda} (x, \bar\gamma)$$
where $\widetilde{\mathcal M} _{1, \lambda} (x, \bar\gamma)$ is $\widetilde{\mathcal M} _{1,\lambda} (\bar\gamma, x)$ with reverse directions.

\item For any $\bar\gamma, \bar\gamma' \in R$,  we define the moduli space 
$$\widetilde{\mathcal M}_1(\bar\gamma, \bar\gamma') = \coprod_{\lambda} \widetilde{\mathcal M}_{1,\lambda}(\bar\gamma, \bar\gamma')$$
where $\widetilde{\mathcal M}_{1,\lambda}(\bar\gamma, \bar\gamma')$ is the moduli space of $\mathbb H_\lambda$-perturbed $\mathbb J_\lambda$-holomorphic strips from $\bar\gamma$ to $\bar\gamma'.$

\item For any $x\in \op{Crit}f$, $\bar\gamma \in R$, we define the moduli space 
$$\widetilde{\mathcal M}_2(x, \bar\gamma) = \coprod_{\lambda}\widetilde{\mathcal M}_{2,\lambda}(x, \bar\gamma)$$
where $\widetilde{\mathcal M}_{2,\lambda}(x, \bar\gamma)$ is the moduli space of tuple $((u_1, \ell_1), (\alpha, T), (u_2,\ell_2))$
such that 
    \be
    \item $(u_1, \ell_1)$ is a $\mathbb H_\lambda$-perturbed $\mathbb J_\lambda$-holomorphic strip with removable singularity at $\infty$ satisfying $\lim_{s\to \infty} \ell_1(s,0)= \lim_{s \to \infty} \ell_1(s,1) \in W^u(x);$
    \item $\alpha: [-T, T] \to M$ is a a Morse trajectory with $\alpha(-T) = \lim_{s\to \infty} \ell_1(s,0) = \lim_{s\to \infty} \ell_1(s,1)$;
    \cpp
    \item  $(u_2, \ell_2)$ is a $J$-holomorphic strip with removable singularity at $-\infty$ satisfying $\lim_{s\to -\infty} \ell_2(s,0)= \lim_{s \to -\infty} \ell_2(s,1) = \alpha(T),$ and\linebreak $\lim_{s\to \infty} (u_2(s,\cdot), \ell_2(s,\cdot)) = \bar\gamma.$ \cb
    \ee

\item For any $x\in \op{Crit}f$, $\bar\gamma \in R$, we define the moduli space 
$$\widetilde{\mathcal M}_2(\bar\gamma, x) = \coprod_{\lambda}\widetilde{\mathcal M}_{2,\lambda}(\bar\gamma, x)$$
where $\widetilde{\mathcal M}_{2,\lambda}(\bar\gamma, x)$ is $\widetilde{\mathcal M}_{2,\lambda}(x, \bar\gamma)$ with reverse directions.

\cpp
\item For any $\bar\gamma, \bar\gamma' \in R$, we define the moduli space 
$$\widetilde{\mathcal M}_2(\bar\gamma, \bar\gamma') = \coprod_{\lambda}\widetilde{\mathcal M}_{2,\lambda}(\bar\gamma, \bar\gamma')$$
where $\widetilde{\mathcal M}_{2,\lambda}(\bar\gamma, \bar\gamma')$ is the moduli space of tuple $((u_1, \ell_1), (\alpha, T), (u_2,\ell_2))$
such that 
    \be
    \item $(u_1, \ell_1)$ is a $J$-holomorphic strip with removable singularity at $+\infty$ and satisfies $\lim_{s\to -\infty} (u_1(s,\cdot), \ell_1(s,\cdot)) = \bar\gamma.$
    \item $\alpha: [-T, T] \to M$ is a a Morse trajectory with $\alpha(-T) = \lim_{s\to \infty} \ell_1(s,0) = \lim_{s\to \infty} \ell_1(s,1)$;
    \item  $(u_2, \ell_2)$ is a $\mathbb H_\lambda$-perturbed $\mathbb J_\lambda$-holomorphic strip with removable singularity at $-\infty$ satisfying $\lim_{s\to -\infty} \ell_2(s,0)= \lim_{s \to -\infty} \ell_2(s,1) = \alpha(T),$ and $\lim_{s\to \infty} (u_2(s,\cdot), \ell_2(s,\cdot)) = \bar\gamma'.$
    \ee
\item For any $\bar\gamma, \bar\gamma' \in R$, we define the moduli space 
$$\widetilde{\mathcal M}_3(\bar\gamma, \bar\gamma') = \coprod_{\lambda}\widetilde{\mathcal M}_{3,\lambda}(\bar\gamma, \bar\gamma')$$
where $\widetilde{\mathcal M}_{3,\lambda}(\bar\gamma, \bar\gamma')$ is $\widetilde{\mathcal M}_{2,\lambda}(\bar\gamma, \bar\gamma')$ with reverse directions. \cb
\end{enumerate}

We define the chain homotopy $K: CF^*(\iota) \to CF^{*-1}(\iota)$ by 
for any $x \in \op{Crit}f$,
\begin{align*}
    K(x) = & \sum_{x' \in \text{Crit}f, \ind(x) - \ind(x') = 1 } \sharp \widetilde{\mathcal M}  (x',x) \cdot x' \\
    & + \sum _{\bar\gamma \in R, \ind(x) - \ind(\bar\gamma) = 1} (\sharp \widetilde{\mathcal M}_1  ( \bar\gamma, x)  + \sharp \widetilde{\mathcal M}_2  (\bar\gamma, x)) \cdot \bar\gamma \\
\end{align*}
and for any $\bar\gamma \in R$,
\cpp
\begin{align*}
    K(\bar\gamma) =  & \sum_{\bar\gamma' \in R, \ind(\bar\gamma) - \ind(\bar\gamma') = 1 } (\sharp \widetilde{\mathcal M}_1(\bar\gamma', \bar\gamma) + \sharp \widetilde{\mathcal M}_2(\bar\gamma', \bar\gamma) + \sharp \widetilde{\mathcal M}_3(\bar\gamma', \bar\gamma)) \cdot \bar\gamma' \\
    & + \sum_{x' \in \text{Crit}f, \ind(\bar\gamma) - \ind(x') = 1 } (\sharp (\widetilde{\mathcal M}_1  (x', \bar\gamma)) + \sharp \widetilde{\mathcal M}_2  (x', \bar\gamma)) \cdot x'.
\end{align*}
\cb

To verify the Equation~\ref{eqn: PHP}, one can check it component wise.
We look at the component from $\bold C$ to $\bold C$ first.
For any $x, x' \in \op{Crit}f$ with $\ind(x) = \ind(x')$ we look at the boundary of the $1$-dimensional moduli space $\widetilde{\mathcal M}(x, x')$ defined in (M1). 
Bad degenerations (holomorphic strips with trees of holomorphic discs attached) are ruled out for index reasons as in the proof of Lemma~\ref{prop:d-is-defined.}.
When $\lambda \to +\infty$, the moduli space corresponds to the $\langle \Phi \circ \Psi(x), x'\rangle$ and when $\lambda \to -\infty$, the holomorphic strip $(u, \ell)$ has to be constant because the Lagrangian is exact and the there is no branch jump. Hence, we get a Morse trajectory connecting $x$ and $x'$. Since $\ind(x) = \ind(x')$, the Morse trajectory has to be constant, and $x = x'$. When $\lambda$ is finite, all the degenerations corresponds to terms in $[K, d]$ shown as in Figure~\ref{chain homotopy M1}.

\cpp
Similarly, the terms of equation \eqref{eqn: PHP} corresponding to the components $\bold R\to\bold R$, $\bold C\to\bold R$, and $\bold R\to\bold C$ can be identified with boundary components of the moduli spaces $\widetilde{\mathcal M}_1(\bar\gamma, \bar\gamma')$,  $\widetilde{\mathcal M}_2(\bar\gamma, \bar\gamma')$,  $\widetilde{\mathcal M}_3(\bar\gamma, \bar\gamma')$, $\widetilde{\mathcal M}_2(x,\bar\gamma)$, and $\widetilde{\mathcal M}_2(\bar\gamma, x)$.
We leave it to the reader to work through these terms on their own, but we'll point out one interesting subtlety: these moduli spaces have additional boundary components that do not correspond to terms in \eqref{eqn: PHP}.
Instead, these additional boundary components occur in pairs, so they cancel with each other because we are using $\Z_2$-coefficients.
An example of such a pair is the following.
Let the parameter $T$ in (M5) go to zero, so that the discs $u_1$ and $u_2$ touch each other.
This also appears as a degeneration (boundary component) of the moduli space (M3) when the strip $u$ breaks off another holomorphic strip on the $+\infty$ side (i.e., becomes a broken trajectory).

\cb

\cbu 
\noindent 
{\bf Step 2. $\Psi \circ \Phi $ is homotopic to $\op{Id}_{CF^*_H}$.}
We define the moduli space $$\widetilde{\mathcal M}^H(\bar\gamma', \bar\gamma) := \coprod_{T \in (-\infty, \infty)} \widetilde{\mathcal M}^H_T(\bar\gamma', \bar\gamma),$$  where $\widetilde{\mathcal M}^H_T(\bar\gamma', \bar\gamma)$ is the space of tuples $((u_1,\ell_1), \alpha, (u_2, \ell_2))$ such that
\be 
 \item $(u_1,\ell_1) \in \widetilde{\mathcal M}_{\bold{\bar J, \bar H}}(\bar\gamma', \emptyset)$,
 \item $\alpha: [-T, T] \to M$ is a Morse trajectory with $\alpha(-T) = \lim_{s \to \infty} \ell_1(s,0) = \lim_{s \to \infty} \ell_1(s,1),$
 \item $(u_2,\ell_2) \in \widetilde{\mathcal M}_{\bold{J, H}}(\emptyset,\bar\gamma)$ with $\alpha(T) = \lim_{s \to -\infty} \ell_2(s,0) = \lim_{s \to -\infty} \ell_2(s,1).$ 
\ee

We define the map $\Xi: CF^*_H(\iota) \to CF^*_H(\iota)$ by 
$$\Xi(\bar\gamma) = \sum_{\ind{\bar\gamma'}= \ind{\gamma}}\widetilde{\mathcal M}^H_0(\bar\gamma', \bar\gamma) \bar\gamma',$$
and the map $K_1: CF^*_H(\iota) \to CF^*_H(\iota)$ by 
$$K_1(\bar\gamma) = \sum_{\ind{\bar\gamma'}= \ind{\gamma}-1}\widetilde{\mathcal M}^H(\bar\gamma', \bar\gamma) \bar\gamma'.$$
\begin{claim}
$\Psi_{\bold C} \circ \Phi_{\bold C} + \Psi_{\bold R, 2} \circ \Phi_{\bold R, 1} + \Psi_{\bold R, 1} \circ \Phi_{\bold R, 2} + \Xi = [K_1, d_H].$ 
\end{claim}
\begin{proof}[Sketch of proof]
The relation is given by the boundary of the $1$-dimensional moduli space $\widetilde{\mathcal M}^H(\bar\gamma', \bar\gamma)$ with $\ind \bar\gamma' = \ind \bar\gamma.$ See Figure~\ref{chain homotopy x}.
\end{proof}

Since $\Phi \circ \Psi = \Psi_{\bold C} \circ \Phi_{\bold C} + \Psi_{\bold R, 2} \circ \Phi_{\bold R, 1} + \Psi_{\bold R, 1} \circ \Phi_{\bold R, 2} + \Psi_{\bold R, 1} \circ \Phi_{\bold R, 1} + \Psi_{\bold R, 2} \circ \Phi_{\bold R, 2}$ and $\Psi_{\bold R, 2} \circ \Phi_{\bold R, 2} = 0$ by the exactness of  $\iota$, in order to show that $\Phi \circ \Psi$ is homotopic to $\op{Id}_{CF^*_H}$ we only need 
\begin{claim}
$\Xi + \Psi_{\bold R, 1} \circ \Phi_{\bold R, 1}$ is homotopic to $\op{Id}_{CF^*_H}$.
\end{claim}
\begin{proof}[Sketch of proof]
We define a glued Hamiltonian $\mathbb H^\lambda := \bar{\bold{H}} \sharp_\lambda {\bold{H}}$ the same way as $\mathbb H_\lambda$ but with $\bold H$ and $\bar{\bold H}$ switched. Similarly, we define $\mathbb J^\lambda.$
We consider the moduli space $\widetilde{\mathcal M}^{\mathbb H}(\bar\gamma', \bar\gamma) = \coprod_{\lambda} \widetilde{\mathcal M}_{\mathbb J^\lambda, \mathbb H^\lambda}(\bar\gamma', \bar\gamma)$,
and construct the map 
$K_2: CF^*_H(\iota) \to CF^*_H(\iota)$ by 
$$K_2(\bar\gamma) = \sum_{\ind{\bar\gamma'}= \ind{\gamma}-1}\widetilde{\mathcal M}^{\mathbb H}(\bar\gamma', \bar\gamma) \bar\gamma'.$$
Then one can show that the boundary of the $1$-dimensional moduli space $\widetilde{\mathcal M}^{\mathbb H}(\bar\gamma', \bar\gamma)$ with $\ind \bar\gamma'  = \ind \bar\gamma$ yields
$$\Xi + \Psi_{\bold R, 1} \circ \Phi_{\bold R, 1}= [K_1, d_H].$$
\end{proof}
\end{proof}

\cbu 
\begin{figure}[!htb] \label{chain homotopy M1}
\centering
\begin{tikzpicture}[scale = 0.7, every node/.style={scale=0.6}]
\def \u {7}
\def \v {2}
\def \s {4}
\def \t {6}
\def \i {8}
\draw[->] (0,1) -- (2,1);
\draw (2,1) -- (3,1);
\draw (3, 1) .. controls (3, 1.5) and (5,1.5) .. (5.5, 1.5)
(3,1) .. controls (3,0.5) and (5, 0.5) .. (5.5,0.5)
(5.5,1.5) .. controls (6, 1.5) and (8,1.5) .. (8,1)
(5.5, 0.5) .. controls (6, 0.5) and (8, 0.5) .. (8,1);
\draw[->] (8,1) -- (9,1);
\draw (9,1) -- (11,1);
\draw[dashed] (5.5, 1.5) -- (5.5, 0.5);
\node[right] at (5.5,1) {$\bar\gamma$};
\node[left] at (-0.2,1) {$x$};
\node[right] at (11.2, 1){$x'$ \hspace{2cm} $\lambda = +\infty$};
\filldraw (0,1) circle [radius = 0.05]
(11,1) circle [radius = 0.05];

\draw[->] (0,1- \v) -- (1,1- \v);
\draw[->] (1,1- \v) -- (2.5,1- \v);
\draw (2.5,1- \v) -- (3,1- \v);
\draw (3, 1- \v) .. controls (3, 1.5- \v) and (5,1.5- \v) .. (5.5, 1.5- \v)
(3,1- \v) .. controls (3,0.5- \v) and (5, 0.5- \v) .. (5.5,0.5- \v)
(5.5,1.5- \v) .. controls (6, 1.5- \v) and (8,1.5- \v) .. (8,1- \v)
(5.5, 0.5- \v) .. controls (6, 0.5- \v) and (8, 0.5- \v) .. (8,1- \v);
\draw[->] (8,1- \v) -- (9,1- \v);
\draw (9,1- \v) -- (11,1- \v);
\node[above] at (1.5, 1- \v) {$x''$};
\node[left] at (-0.2,1- \v) {$x$};
\node[right] at (11.2, 1- \v){$x'$};
\filldraw (0,1- \v) circle [radius = 0.05]
(1.5,1- \v) circle [radius = 0.05]
(11,1- \v) circle [radius = 0.05];

\draw[->] (0,1- \s) -- (2,1- \s);
\draw (2,1- \s) -- (3,1- \s);
\draw (3, 1- \s) .. controls (3, 1.5- \s) and (5,1.5- \s) .. (5.5, 1.5- \s)
(3,1- \s) .. controls (3,0.5- \s) and (5, 0.5- \s) .. (5.5,0.5- \s)
(5.5,1.5- \s) .. controls (6, 1.5- \s) and (8,1.5- \s) .. (8,1- \s)
(5.5, 0.5- \s) .. controls (6, 0.5- \s) and (8, 0.5- \s) .. (8,1- \s);
\draw[->] (8,1- \s) -- (9,1- \s);
\draw[->] (9,1- \s) -- (10.5,1- \s);
\draw (10.5,1- \s) -- (11,1- \s);
\node[above] at (9.5, 1- \s) {$x''$};
\node[left] at (-0.2,1- \s) {$x$};
\node[right] at (11.2, 1- \s){$x'$};
\filldraw (0,1- \s) circle [radius = 0.05]
(9.5,1- \s) circle [radius = 0.05]
(11,1- \s) circle [radius = 0.05];

\draw[->] (0,1- \t ) -- (2,1- \t );
\draw (2,1- \t ) -- (3,1- \t );
\draw (3, 1- \t ) .. controls (3, 1.5- \t ) and (4, 1.7-\t) .. (4.5, 1- \t )
      (3, 1- \t ) .. controls (3, 0.5- \t ) and (4, 0.3- \t ) .. (4.5, 1- \t )
      (4.5, 1 - \t ) .. controls (5.5, 1.5- \t ) and (8, 1.8 - \t ) .. (8,1- \t )
      (4.5, 1 - \t ) .. controls (5.5, 0.5- \t ) and (8, 0.2 - \t ) .. (8,1- \t );
\draw[->] (8,1- \t ) -- (9,1- \t );
\draw (9,1- \t ) -- (11,1- \t );
\node[left] at (-0.2,1- \t ) {$x$};
\node[right] at (11.2, 1- \t ){$x'$};
\node[above] at (4.5, 1 - \t ){$\bar\gamma$};
\filldraw (0,1- \t ) circle [radius = 0.05]
(11,1- \t ) circle [radius = 0.05];
\draw (4.5, 1 - \t ) circle [radius = 0.05];
\fill [color=white!60] (4.5, 1 - \t ) circle [radius = 0.05];

\draw[->] (0,1- \i ) -- (2,1- \i );
\draw (2,1- \i ) -- (3,1- \i );
\draw (11-3, 1- \i ) .. controls (11-3, 1.5- \i ) and (11-4, 1.7-\i) .. (11-4.5, 1- \i )
      (11-3, 1- \i ) .. controls (11-3, 0.5- \i ) and (11-4, 0.3- \i ) .. (11-4.5, 1- \i )
      (11-4.5, 1 - \i ) .. controls (11-5.5, 1.5- \i ) and (11-8, 1.8 - \i ) .. (11-8,1- \i )
      (11-4.5, 1 - \i ) .. controls (11-5.5, 0.5- \i ) and (11-8, 0.2 - \i ) .. (11-8,1- \i );
\draw[->] (8,1- \i ) -- (9,1- \i );
\draw (9,1- \i ) -- (11,1- \i );
\node[left] at (-0.2,1- \i ) {$x$};
\node[right] at (11.2, 1- \i ){$x'$};
\node[above] at (11-4.5, 1 - \i ){$\bar\gamma$};
\filldraw (0,1- \i ) circle [radius = 0.05]
(11, 1- \i ) circle [radius = 0.05];
\draw (11-4.5, 1 - \i ) circle [radius = 0.05];
\fill [color=white!60] (11-4.5, 1 - \i ) circle [radius = 0.05];

\draw[->] (0,-2- \u) -- (2,-2- \u);
\draw (2,-2- \u) -- (3,-2- \u);
\draw (3, -2- \u) .. controls (3, -1.5- \u) and (5,-1.5- \u) .. (5.5, -1.5- \u)
(3,-2- \u) .. controls (3,-2.5- \u) and (5, -2.5- \u) .. (5.5,-2.5- \u)
(5.5,-1.5- \u) .. controls (6, -1.5- \u) and (8,-1.5- \u) .. (8,-2- \u)
(5.5, -2.5- \u) .. controls (6, -2.5- \u) and (8, -2.5- \u) .. (8,-2- \u);
\draw[->] (8,-2- \u) -- (9,-2- \u);
\draw (9,-2- \u) -- (11,-2- \u);
\node[left] at (-0.2,-2- \u) {$x$};
\node[right] at (11.2, -2- \u){$x'$  \hspace{2cm} $\lambda = -\infty$};
\node[right] at (5,-2- \u) {constant};
\filldraw (0,-2- \u) circle [radius = 0.05]
(11,-2- \u) circle [radius = 0.05];

\end{tikzpicture}
\caption[]{Degeneration of $\widetilde{\mathcal M}(x, x')$. The smaller region in the 4th and the 5th pictures are unperturbed $J$-holomorphic curves.}
\label{chain homotopy M1}
\end{figure}

\begin{figure}[!htb]
\centering
\begin{tikzpicture}[scale = 0.7, every node/.style={scale=0.6}]
\def \d {2}
\def \a {0}
\def \b {-0.5}
\def \c {-1}
\def \x {2}
\def \y {4}
\def \z {6}
\def \w {8}
\def \u {10}
\def \v {12}

\draw[dashed] (0* \d, \a) -- (0* \d, \c);
\draw (0* \d, \c) .. controls (1.5 * \d, \c) and (2 * \d, 1.8 * \b) .. (2 * \d, \b)
      (2 * \d, \b) .. controls (2 * \d, 0.2 * \b) and (1.5 * \d, \a) .. (0* \d, \a)
      (3 * \d, \b) --(4 * \d, \b);
\draw[->] (2 * \d, \b) --(3 * \d, \b);
      
\draw[dashed] (6* \d, \a) -- (6* \d, \c);
\draw (6* \d, \c) .. controls (4.5 * \d, \c ) and (4*\d, 1.8* \b) .. (4 * \d, \b)
      (4 * \d, \b) .. controls (4 * \d, 0.2 * \b) and (4.5*\d, \a) .. (6* \d, \a);
\node[right] at (7* \d, \b) {$ $};
\draw[-Triangle, very thick](3 * \d, -1) -- (3 * \d, -2);
\node[right] at (3 * \d, -1.5){degenerates};

\draw[dashed] (0* \d, \a-\x ) -- (0* \d, \c-\x );
\draw (0* \d, \c-\x ) .. controls (1.5 * \d, \c-\x ) and (2 * \d, 1.8 * \b-\x ) .. (2 * \d, \b-\x )
      (2 * \d, \b-\x ) .. controls (2 * \d, 0.2 * \b-\x ) and (1.5 * \d, \a-\x ) .. (0* \d, \a-\x );
\draw[->] (2 * \d, \b-\x ) --(2.5 * \d, \b-\x );
\draw[->] (2.5 * \d, \b-\x ) -- (3.5 * \d, \b-\x );
\draw (3.5 * \d, \b-\x ) --(4 * \d, \b-\x );
      
\draw[dashed] (6* \d, \a-\x ) -- (6* \d, \c-\x );
\draw (6* \d, \c-\x ) .. controls (4.5 * \d, \c -\x ) and (4*\d, 1.8* \b-\x ) .. (4 * \d, \b-\x )
      (4 * \d, \b-\x ) .. controls (4 * \d, 0.2 * \b-\x ) and (4.5*\d, \a-\x ) .. (6* \d, \a-\x );

\filldraw (3 * \d, \b-\x ) circle [radius = 0.05];
\node[right] at (7* \d, \b-\x ) {$\Psi_{\bold C} \circ \Phi_{\bold C}$};

\draw[dashed] (0* \d, \a-\y ) -- (0* \d, \c-\y );
\draw (0* \d, \c-\y ) .. controls (1.5 * \d, \c-\y ) and (2 * \d, 1.8 * \b-\y ) .. (2 * \d, \b-\y )
      (2 * \d, \b-\y ) .. controls (2 * \d, 0.2 * \b-\y ) and (1.5 * \d, \a-\y ) .. (0* \d, \a-\y )
      (3 * \d, \b-\y ) --(4 * \d, \b-\y );
\draw[->] (2 * \d, \b-\y ) --(3 * \d, \b-\y );
      
\draw[dashed] (6* \d, \a-\y ) -- (6* \d, \c-\y );
\draw (6* \d, \c-\y ) .. controls (5.5 * \d, \c -\y )  .. (5 * \d, \b-\y )
      (6* \d, \a - \y) .. controls (5.5 * \d, \a -\y) ..  (5 * \d, \b-\y )
      (4 * \d, \b-\y ) .. controls (4 * \d, 0.2 * \b-\y ) and (4.5*\d, \a-\y ) .. (5 * \d, \b-\y )
      (4 * \d, \b -\y) .. controls (4 * \d, 1.8 * \b-\y ) and (4.5*\d, \c-\y ) .. (5 * \d, \b-\y );
\filldraw [color=white!60] (5 * \d, \b-\y ) circle [radius = 0.05];
\draw (5 * \d, \b-\y ) circle [radius = 0.05];
\node[right] at (7* \d, \b-\y ) {$\Psi_{\bold{R},1} \circ \Phi_{\bold{R},2}$};

\draw[dashed] (0* \d, \a-\z ) -- (0* \d, \c-\z );

\draw  (0* \d, \a -\z) .. controls (0.5 * \d, \a -\z) .. (1 * \d, \b-\z )
        (0* \d, \c-\z ) .. controls (0.5 * \d, \c-\z ) .. (1 * \d, \b-\z )
      (1 * \d, \b-\z ) .. controls (1.5 * \d, \a-\z ) and (2 * \d, 0.2 *\b -\z ) .. (2* \d, \b-\z )
      (1 * \d, \b-\z ) .. controls (1.5 * \d, \c-\z ) and (2 * \d, 1.8 *\b -\z ) .. (2* \d, \b-\z );
\draw[->] (2 * \d, \b-\z ) --(3 * \d, \b-\z );
\draw (3 * \d, \b-\z ) --(4 * \d, \b-\z );

\draw[dashed] (6* \d, \a-\z) -- (6* \d, \c-\z);
\draw (6* \d, \c-\z) .. controls (4.5 * \d, \c -\z) and (4*\d, 1.8* \b-\z) .. (4 * \d, \b-\z)
      (4 * \d, \b-\z) .. controls (4 * \d, 0.2 * \b-\z) and (4.5*\d, \a-\z) .. (6* \d, \a-\z);
\filldraw [color=white!60] (1 * \d, \b-\z ) circle [radius = 0.05];
\draw (1 * \d, \b-\z ) circle [radius = 0.05];
\node[right] at (7* \d, \b-\z ) {$\Psi_{\bold{R},2} \circ \Phi_{\bold{R},1}$};

\draw[dashed] (0* \d, \a-\w) -- (0* \d, \c-\w);
\draw (0* \d, \c-\w) .. controls (1.5 * \d, \c-\w) and (3 * \d, 1.8 * \b-\w) .. (3 * \d, \b-\w)
      (3 * \d, \b-\w) .. controls (3 * \d, 0.2 * \b-\w) and (1.5 * \d, \a-\w) .. (0* \d, \a-\w);

\draw[dashed] (6* \d, \a-\w) -- (6* \d, \c -\w);
\draw (6* \d, \c-\w) .. controls (4.5 * \d, \c -\w) and (3*\d, 1.8* \b-\w) .. (3 * \d, \b-\w)
      (3 * \d, \b-\w) .. controls (3 * \d, 0.2 * \b-\w) and (4.5*\d, \a-\w) .. (6* \d, \a-\w);
\node[right] at (7* \d, \b-\w) {$\Xi$};

\node[right] at (7* \d, \b-\y ) {$\Psi_{\bold{R},1} \circ \Phi_{\bold{R},2}$};

\draw[dashed] (0* \d, \a-\u) -- (0* \d, \c-\u);
\draw[dashed] (1 * \d, \a-\u) -- (1 * \d, \c-\u);
\draw (0* \d, \c-\u) -- (1 * \d, \c-\u);
\draw (0* \d, \a-\u) -- (1 * \d, \a-\u);
\draw (1* \d, \c-\u) .. controls (1.5 * \d, \c-\u) and (2 * \d, 1.8 * \b-\u) .. (2 * \d, \b-\u)
      (2 * \d, \b-\u) .. controls (2 * \d, 0.2 * \b-\u) and (1.5 * \d, \a-\u) .. (1 * \d, \a-\u)
      (3 * \d, \b-\u) --(4 * \d, \b-\u);
\draw[->] (2 * \d, \b-\u) --(3 * \d, \b-\u);
      
\draw[dashed] (6* \d, \a-\u) -- (6* \d, \c-\u);
\draw (6* \d, \c-\u) .. controls (4.5 * \d, \c -\u) and (4*\d, 1.8* \b-\u) .. (4 * \d, \b-\u)
      (4 * \d, \b-\u) .. controls (4 * \d, 0.2 * \b-\u) and (4.5*\d, \a-\u) .. (6* \d, \a-\u);
\node[right] at (7* \d, \b-\u) {$d_H \circ K_1$};

\draw[dashed] (0* \d, \c-\v ) -- (0* \d, \a-\v );
\draw (0* \d, \c-\v ) .. controls (1.5 * \d, \c-\v ) and (2 * \d, 1.8 * \b-\v ) .. (2 * \d, \b-\v )
      (2 * \d, \b-\v ) .. controls (2 * \d, 0.2 * \b-\v ) and (1.5 * \d, \a-\v ) .. (0* \d, \a-\v )
      (3 * \d, \b-\v ) --(4 * \d, \b-\v );
\draw[->] (2 * \d, \b-\v ) --(3 * \d, \b-\v );

\draw[dashed] (6* \d, \a-\v) -- (6* \d, \c-\v);
\draw (5* \d, \c-\v) .. controls (4.5 * \d, \c -\v) and (4*\d, 1.8* \b-\v) .. (4 * \d, \b-\v)
      (4 * \d, \b-\v) .. controls (4 * \d, 0.2 * \b-\v) and (4.5*\d, \a-\v) .. (5* \d, \a-\v)
      (5* \d, \c-\v) -- (6* \d, \c-\v)
      (5* \d, \a-\v) -- (6* \d, \a-\v);
\draw[dashed] (5* \d, \c-\v) --  (5* \d, \a-\v);
\node[right] at (7* \d, \b-\v) {$K_1 \circ d_H$};



\end{tikzpicture}
\caption[]{Degeneration of $\widetilde{\mathcal M}^H(\bar\gamma', \bar\gamma)$.}
\label{chain homotopy x}
\end{figure}

\cb

\section{Examples}\label{example}
We calculate the Floer cohomology of some immersed Lagrangian spheres inside the smoothing of an $A_N$ surface.
These calculations first appeared in the preprint \cite{2013arXiv1311.2327A} by the first author.
We present here the main points and refer the reader to the paper for more details.

We start by introducing the main objects.
For any integer $N\geq1$, let $F_N:\C^3\to\C$ be the function $F_N=x_1x_2-(x_3-1)\cdots (x_3-N)$.
We consider the symplectic manifold
\begin{displaymath}
  M_N=\sett{ F_N=0 } \subset \C^3,\quad \omega = \frac{i}{2} \sum d x_j\wedge d\bar x_j\bigl|M_N.
\end{displaymath}
We make $M_N$ an exact symplectic manifold by choosing the 1-form
\begin{displaymath}
  \lambda = \frac{i}{4}\sum\left( x_j d\bar x_j - \bar x_j d x_j\right)
\end{displaymath}
as the primitive for $\omega$ (i.e., $d\lambda = \omega$).
We define a squared phase map on $M_N$ by choosing for a holomorphic volume form $\Omega$ the Poincar\'e residue of $dx_1\wedge dx_2\wedge dx_3 / F_N$.
This complex volume form defines a squared phase map $\mathrm{Det}^2_\Omega:\mathcal G\to S^1$ (where $\mathcal G$ is the bundle of Lagrangian planes over $M_N$).

We give $M_N$ two additional structures which are useful for constructing Lagrangian submanifolds.
The first additional structure is the map $\pi:M_N\to\C$, $(x_1,x_2,x_3)\mapsto x_3$, which is a Lefschetz fibration with critical values $1,2,\ldots,N\in\C$.
The smooth fibers are of the form $x_1x_2=c$ with $c\neq 0$, and the critical fibers are of the form $x_1x_2=0$.
The second additional structure consists of real-valued maps $H=\frac{1}{2}|x_1|^2-\frac{1}{2}|x_2|^2$ and $G=|x_3|^2$.
$H$ and $G$ Poisson commute and hence define an integrable system on $M_N$.
The leaves of the integrable system are compact, so the smooth leaves are Lagrangian tori.

\begin{figure}
  \includegraphics[]{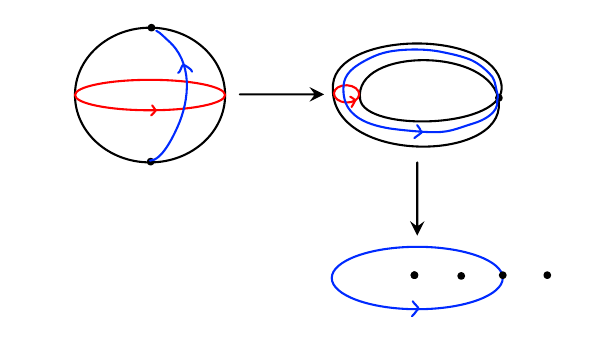}
  \caption{Top left: The sphere $S^2$. The red (horizontal) circle is a line of latitude ($a=\text{constant}$ in cylindrical coordinates $(a,e^{ib})$). The blue (vertical) curve is a line of longitude ($e^{ib}=\text{constant}$).\\\\Top right: The immersed Lagrangian $\iota_{N,r}$. The red (small) circle is the image of the line of latitude and the blue (large) circle is the image of the line of longitude. The line of longitude starts on the south pole ($q_{S^2}$) branch and ends on the north pole ($p_{S^2}$) branch.\\\\Bottom right: The circle $|z|=r$ inside $\C$ which is the image of the Lagranian under the Lefschetz fibration $\pi$. The black dots are the points $0,1,\ldots,N$. The line of longitude wraps counterclockwise around the circle. The point $r\in\C$ is the image of both the north and south poles, and points near the south pole map to the first quadrant while points near the north pole map to the fourth quadrant.}
\label{fig:immersed-sphere}
\end{figure}
For each $N\geq 1$ and each $r\in \sett{ 1, \ldots, N}$ we define an immersed Lagrangian sphere $\iota_{N,r}:S^2\to M_N$ inside $M_N$.
See Figure \ref{fig:immersed-sphere}.
The image $L_{N,r}=\iota_{N,r}(S^2)$ can be described using either the Lefschetz fibration or the integrable system.
In terms of the integrable system, we define
\begin{displaymath}
  L_{N,r}=\sett{ H=0,\ G=r^2}.
\end{displaymath}
In terms of the Lefschetz fibration, $L_{N,r}$ is the union of all the vanishing cycles over the curve $|z|=r$ in the base $\C$ of the fibration.
Over the point $z\in \C$, the fiber of the Lefschetz fibration is $\sett{x_1x_2=(z-1)\cdots(z-N),x_3=z}$, and the vanishing cycle is the set of points with $|x_1|=|x_2|$.
The union of these vanishing cycles is clearly the set of points with $G=r^2$ and $H=0$, so the two descriptions are equivalent.
The critical point of $\pi$ in the fiber over $z=r$ is also a critical point of the integrable system and it has focus-focus type.
Thus $L_{N,r}$ has a transverse self-intersection at this critical point.
The vanishing cycles are all circles, so topologically $L_{N,r}$ is a singular circle bundle over the base curve $|z|=r$ that has one fiber pinched to a point, so $L_{N,r}$ is a ``pinched torus''.
This also follows from the fact that $L_{N,r}$ contains a unique singularity of the integrable system and it is of focus-focus type.

More formally, there exists a Lagrangian immersion $\iota_{N,r}:S^2\to M_N$ with image $L_{N,r}$, and we can think of the north and south poles as being the preimage of the double point (pinched circle) in $L_{N,r}$.
To be explicit, let $S^2\subset\R^3$ be the standard sphere.
We define cylindrical coordinates $(a,e^{ib})\in(-\pi,\pi)\times S^1$ on $S^2$ by
\begin{align*}
  x &= \cos(a/2)\cos(b) \\
  y &= \cos(a/2)\sin(b) \\
  z &= \sin(a/2).
\end{align*}
We let $p_{S^2}=(\pi,0)$ (the ``north pole'') and $q_{S^2}=(-\pi,0)$ (the ``south pole'').
We want to give an explicit formula for $\iota_{N,r}$ in terms of cylindrical coordinates.
To this end, let $\rho=x^2+y^2=\cos^2(a/2)$ and choose a function $f:[-\pi,\pi]\to[-\pi,\pi]$ such that
\begin{gather*}
  f(\pm\pi)=\pm\pi,\\
  f'(a)>0\,\mathrm{for}\,a\neq0,\\
  f(a)=-\pi+\rho\,\mathrm{for}\,a\,\mathrm{near}\,-\pi,\\
  f(a)=a\,\mathrm{for}\,a\,\mathrm{near}\,0,\\
  f(a)=\pi-\rho\,\mathrm{for}\,a\,\mathrm{near}\,\pi.\\
\end{gather*}
Also, let $\xi:[-\pi,\pi]\to\C$ be the function
\begin{displaymath}
  \xi(a)=\sqrt{-re^{if(a)}-1}\cdots\sqrt{-re^{if(a)}-r}\cdots\sqrt{-re^{if(a)}-N}.
\end{displaymath}
For $a=0$, all the terms under the square root signs are negative, and we choose the square root function so that (at $a=0$) $\sqrt{x}=i\sqrt{|x|}$ for $x<0$.
In other words, we think of the phase of the negative real line as being $\pi$ and put the branch cut along the positive real axis.
Then our formula for $\iota_{N,r}$ in terms of cylindrical coordinates is
\begin{displaymath}
  \iota_{N,r}:(a,e^{ib})\mapsto (e^{ib}\xi(a),e^{-ib}\xi(a),-re^{if(a)})\in M_N\subset\C^3.
\end{displaymath}
It is proved in Lemma 2.5 in \cite{2013arXiv1311.2327A} that this extends smoothly to all of $S^2$ and gives an immersion of $S^2$ with image $L_{N,r}$ that maps the north and south poles to the double point.

The image in $\C$ under $\pi\circ\iota_{N,r}$ of the path $s\mapsto(s,e^{ib})\in S^2$ that goes from the south pole to the north pole is the circle that starts at $z=r$ and travels counterclockwise until the starting point $z=r$ is again reached.

A grading for $\theta_{N,r}:S^2\to\R$ is given by the formula
\begin{equation}\label{eq:grading immersed sphere}
  \theta_{N,r}(a,e^{ib})=\frac{f(a)}{\pi}.
\end{equation}
See Definition 2.8 in \cite{2013arXiv1311.2327A} for details.
To calculate the index of the self-intersection point, we need to find the tangent spaces of the immersion at the north and south poles.
This can be accomplished by calculating $\pd{}{a}\iota_{N,r}$ near $a=\pm\pi$ and rescaling to get a finite vector (which depends on $e^{ib}$).
\begin{lemma}
  There exists constants $C_+$ and $C_-$ with $C_+/C_-\in\R\setminus0$ so that
  \begin{displaymath}
    \lim_{a\to\pm\pi} \rho^{1/2}\pd{}{a}\iota_{N,r}=(e^{ib}C_{\pm}e^{i(\pi/2 \pm \pi/4)}, e^{-ib}C_{\pm}e^{i(\pi/2 \pm \pi/4)},0)\in\C^3.
  \end{displaymath}
\end{lemma}
\begin{proof}
  First note that $\rho(a)=\cos^2(a/2)$ and $f'(a)=\mp\rho'(a)=\mp\cos(a/2)\sin(a/2)$ near $a=\pm\pi$.
  Thus $\rho^{1/2}f'(a)=\mp\cos^2(a/2)\sin(a/2)\to 0$ as $a\to\pm\pi$.

  Next we break $\xi(a)$ up into three factors: $\xi(a)=\xi_0(a)\xi_1(a)\xi_2(a)$ with
  \begin{align*}
    \xi_0(a)&=\sqrt{-re^{if(a)}-1}\cdots\sqrt{-re^{if(a)}-(r-1)},\\
    \xi_1(a)&=\sqrt{-re^{if(a)}-r},\\
    \xi_0(a)&=\sqrt{-re^{if(a)}-(r+1)}\cdots\sqrt{-re^{if(a)}-N}.\\    
  \end{align*}
  Since $\xi_0(a)\neq0$ and $\xi_2(a)\neq0$ for all $a$, $\rho^{1/2}f'(a)\to 0$, and $\xi_1(a)\to0$, we have
  \begin{displaymath}
    \lim_{a\to\pm\pi}\rho^{1/2}\xi'(a)=\xi_0(\pm\pi)\xi_2(\pm\pi)\lim_{a\to\pm\pi}\rho^{1/2}\xi_1'(a).    
  \end{displaymath}

  We turn to calculating $\rho^{1/2}\xi_1'(a)$.
  Near $a=\pm\pi$,
  \begin{align*}
    -re^{if(a)}-r&=-re^{i\pm\pi\mp i\rho}-r=r(e^{\mp i\rho}-1)=r(\cos\rho -1 \mp i\sin\rho)\\
                 &=\mp i r\rho + O(\rho^2).
  \end{align*}
  Thus, by the way we defined the square root function,
  \begin{displaymath}
    \sqrt{-re^{if(a)}-r}=\sqrt{\mp i r\rho + O(\rho^2)}=e^{i(\pi/2 \pm \pi/4)}\sqrt{r\rho}(1+O(\rho)).
  \end{displaymath}
  Hence we have
  \begin{displaymath}
    \xi_1'(a)=e^{i(\pi/2 \pm \pi/4)}\sqrt{r}\rho^{-1/2}/2+O(\rho^{1/2})
  \end{displaymath}
  and
  \begin{displaymath}
    \rho^{1/2}\xi_1'(a)=e^{i(\pi/2 \pm \pi/4)}\sqrt{r}/2+O(\rho).
  \end{displaymath}
  It follows that
  \begin{displaymath}
    \lim_{a\to\pm\pi}\rho^{1/2}\xi_1'(a)=e^{i(\pi/2 \pm \pi/4)}\sqrt{r}/2.
  \end{displaymath}

  Since
  \begin{displaymath}
    \pd{}{a}\iota_{N,r}=(e^{ib}\xi'(a),e^{-ib}\xi'(a),-rf'(a)e^{if(a)}),
  \end{displaymath}
  we have established that
  \begin{displaymath}
    \lim_{a\to\pm\pi} \rho^{1/2}\pd{}{a}\iota_{N,r}=(e^{ib}C'_{\pm}e^{i(\pi/2 \pm \pi/4)}\sqrt{r}/2,e^{-ib}C'_{\pm}e^{i(\pi/2 \pm \pi/4)}\sqrt{r}/2,0).
  \end{displaymath}
  where $C'_{\pm}=\xi_0(\pm\pi)\xi_2(\pm\pi)$.
  Let $C_{\pm}=C'_{\pm}\sqrt{r}/2$.
  To finish the proof, we note that $\xi_2(\pi)=\xi_2(-\pi)\neq0$ and $\xi_0(\pi)=(-1)^{r-1}\xi_0(-\pi)\neq0$, so $C_+/C_-=(-1)^{r-1}$.
\end{proof}
\begin{cor}
  At the double point of $L_{N,r}$, multiplication by $i$ maps the tangent space of one branch into the tangent space of the other branch.
\end{cor}

\begin{lemma}
$\ind (p_{S^2},q_{S^2})=-1$ and $\ind (q_{S^2},p_{S^2})=3$.
\end{lemma}
\begin{proof}
  By Definition \ref{defn:branch jump index}, the index of a branch jump $(p,q)$ is $2+\theta_{N,r}(q)-\theta_{N,r}(p)-2\mathrm{Angle}(d\iota(T_pL),d\iota(T_qL))$.
  By the previous corollary, $2\mathrm{Angle}=1$ (for both orderings of the branch jump), and $\theta_{N,r}(p_{S^2})=1$, $\theta_{N,r}(q_{S^2})=-1$ by formula \eqref{eq:grading immersed sphere}.
  Thus
  \begin{align*}
    \ind(p_{S^2},q_{S^2})&= 2+\theta_{N,r}(q_{S^2})-\theta_{N,r}(p_{S^2})-2\mathrm{Angle}=2-1-1-1=-1,\\
    \ind(q_{S^2},p_{S^2})&= 2+\theta_{N,r}(p_{S^2})-\theta_{N,r}(q_{S^2})-2\mathrm{Angle}=2+1+1-1=3.
  \end{align*}
\end{proof}

The next step in calculating the Floer cohomology is to find the holomorphic strips with boundary on $L_{N,r}$.
We use the complex structure $J$ on $M_N$ induced by the standard complex structure on $\C^3$, and take $J_t=J$ for all $t$ (so that we have a time-independent complex structure).
Finding holomorphic strips with a time-independent complex structure is equivalent to finding holomorphic discs with two marked points, so we start by describing holomorphic discs.

First, we note that if $u$ is a holomorphic disc with boundary on $L_{N,r}$, then $\pi\circ u$ is a holomorphic map from the unit disc to the disc of radius $r$.
See Figure \ref{fig:immersed-sphere}: the image of $\pi\circ u$ fills the the disc in the lower right of the figure.
Such discs must wind the boundary around strictly in the counterclockwise direction (unless constant), and thus an outgoing point of $u$ has branch jump of type $(p_{S^2},q_{S^2})$ and an incoming point has type $(q_{S^2},p_{S^2})$.
If the disc has just one branch jump then the symplectic area is positive and equal to $\mathcal A(q_{S^2},p_{S^2})$.
Thus $\mathcal A(q_{S^2},p_{S^2})>0$, and $\mathcal A(p_{S^2},q_{S^2})<0$, and it follows that $\iota_{N,r}$ satisfies the strong positivity condition.

Holomorphic discs with branch jumps at the marked points are classified in Proposition 6.2 in \cite{2013arXiv1311.2327A}:
Any such disc is of the form $u=(f,g,rh)$ with $h$ a Blaschke product\footnote{\cpp A Blaschke product is a continuous map from the closed unit disc to the closed unit disc that takes the boundary to the boundary and is holomorphic in the interior. Such a function is a finite product of functions of the form $z\mapsto \frac{z-a}{1-\bar a z}$ with $|a|<1$. The degree of the restricted map $S^1\to S^1$ is equal to the number of terms in the Blaschke product.\cb }
\begin{align*}
  f&= e^{i\theta}\xi_1\cdots\xi_{r-1}\sqrt{h-r}\sqrt{2h-r}\cdots\sqrt{rh-r}\sqrt{rh-(r+1)}\cdots\sqrt{rh-N},\\
  g&= e^{-i\theta}\eta_1\cdots\eta_{r-1}\sqrt{h-r}\sqrt{2h-r}\cdots\sqrt{rh-r}\sqrt{rh-(r+1)}\cdots\sqrt{rh-N}.
\end{align*}
Here $\xi_j,\eta_j$ are Blaschke products with $\xi_j\eta_j=\frac{rh-j}{jh-r}$.
Additional non branch jump points can be inserted anywhere on the boundary.

Any holomorphic strip is of the following form:
Take a disc with one outgoing marked point and one incoming marked point such that at least one of the marked points has a branch jump.
There is then a biholomorphism that maps the strip to the disc minus the two points in such a way that $-\infty$ maps to the incoming marked point and $+\infty$ maps to the outgoing marked point.
Any holomorphic strip can then be factored as such a biholomorphism followed by one of the holomorphic discs described above.
Conversely, any such composition is a holomorphic strip.

The next step is to choose a Morse function on $S^2$.
We take a Morse function on $S^2$ that has a maximum at $p_{max}:=(0,1)$ (the right-hand side refers to cylindrical coordinates $(a,e^{ib})$ on $S^2$) and a minumum at $p_{min}:=(0,-1)$ and no other critical points.
The points $p_{min}$ and $p_{max}$ lie on the equator $\sett{a=0}$ between the north pole $p_{S^2}$ and the south pole $q_{S^2}$.

The Floer cochain complex has generators $(p_{S^2},q_{S^2}),(q_{S^2},p_{S^2}),p_{min},p_{max}$ and the indices of these generators are
\begin{align*}
  \ind(p_{S^2},q_{S^2})&=-1,\\
  \ind(p_{min})&=0,\\
  \ind(p_{max})&=2,\\
  \ind(q_{S^2},p_{S^2})&=3.\\
\end{align*}
For degree reasons, we only need the calculate the terms $\langle d (p_{S^2},q_{S^2}),p_{min}\rangle$ and $\langle d p_{max}, (q_{S^2},p_{S^2})\rangle$.
For the first term we count configurations of type (b) in Figure \ref{fig:four types of trajectories} and for the second term we count configurations of type (c).

In both cases, since there is only one branch jump, the Blaschke product $h$ in the holomorphic disc has degree 1 (is an automorphism).
Since $h$ is an automorphism of the unit disc, the right-hand side in the equation $\xi_j\eta_j=\frac{rh-j}{jh-r}$ is also an automorphism of the unit disc.
So for each $j$ exactly one of $\xi_j$ or $\eta_j$ is a constant, and the other one is the right-hand side divided by the constant.
Thus the choice of factoring $\xi_j\eta_j$ into a constant and nonconstant term amounts to choosing a component of the moduli space.
Since $1\leq j\leq r-1$, there are $2^{r-1}$ ways to choose a factoring, and hence $2^{r-1}$ components of the moduli space.
If we fix the location of the branch jump marked point at $1\in S^1$, then there are two degrees of freedom left in the choice of $h$.
One degree of freedom coincides with the automorphisms of the strip, and will be removed when we quotient by automorphisms to get the moduli space.
The remaining degree of freedom is parameterized by the location of the other marked point on the boundary of the disc, so it is parameterized by $S^{1}\setminus\sett{1}\cong \R$.
With $\mathrm{ev}$ denoting the evaluation map of this other marked point, we have $\pi\circ\iota_{N,r}\circ\mathrm{ev}=rh$, which has image $rS^{1}\setminus\sett{r}\subset\C$.
So varying the location of the marked point causes $\mathrm{ev}$ to vary along the lines of longitude $e^{ib}=\text{constant}$ in $S^2$.
The remaining parameter $e^{i\theta}$ determines where the the marked point maps to on $\pi^{-1}(z)\cap L_{N,r}\cong S^1$ ($|z|=r$).
More precisely, varying $\theta$ causes $\mathrm{ev}$ to vary along the lines of latitude of the form $a=\mathrm{constant}$ in $S^2$.
Putting it all together, we have shown the following:
\begin{lemma}
  The moduli space of strips with one branch jump and one non branch jump has $2^{r-1}$ components.
  Each component is diffeomorphic to $\R\times S^1$, and under the evaluation map of the non branch jump point maps diffeomorphically onto the cylindrical coordinate chart $(-\pi,\pi)\times S^1=S^{2}\setminus\sett{q_{S^2},p_{S^2}}$ of the sphere.
\end{lemma}
It is shown in Section 6.3 of \cite{2013arXiv1311.2327A} that all the strips and discs are regular.

Finally we can calculate the Floer differential.
Consider first the term $\langle d (p_{S^2},q_{S^2}),p_{min}\rangle$.
Recall that this counts trajectories of type (b) in Figure \ref{fig:four types of trajectories}.
These trajectories correspond to the fiber product of the unstable manifold of $p_{min}$ with respect to the negative gradient flow (equivalently, the stable manifold with respect to the positive gradient flow) with the moduli space described above.
The unstable manifold is just $p_{min}$ because $p_{min}$ is the global minimum of the Morse function.
Thus the fiber product has one element for each component of the moduli space, and we have
\begin{displaymath}
  \langle d (p_{S^2},q_{S^2}),p_{min}\rangle = 2^{r-1}.
\end{displaymath}
Similarly, $\langle d p_{max}, (q_{S^2},p_{S^2})\rangle$ is the fiber product of the moduli space with the stable manifold of $p_{max}$ with respect to the negative gradient flow.
The stable manifold is just $p_{max}$, so
\begin{displaymath}
\langle d p_{max}, (q_{S^2},p_{S^2})\rangle = 2^{r-1}.
\end{displaymath}
Finally we have the main result of this subsection:
\begin{thm}
  The Floer cohomology with $\Z_2$-coefficients of $\iota_{N,r}:S^2\to M_N$ is $0$ for $r=1$.
  For $r>1$, it is isomorphic to $\Z_2$ in degrees $-1,0,2,3$ and $0$ in other degrees.
\end{thm}
\begin{proof}
  If $r=1$ then the kernel of $d$ is $0$ so the cohomology is $0$.
  If $r>1$ then $d=0$ and the statement follows from the fact that the cochain complex has one generator (basis vector) in each of the degrees $-1,0,2,3$ and no other generators.
\end{proof}

\bibliographystyle{amsplain}
\bibliography{mybib}

\end{document}